\newif\ifsingleCol
\newif\ifNonArxivVersion
\newif\iffullversion

\NonArxivVersiontrue 
\fullversiontrue

\ifsingleCol
    \documentclass[12pt,draftcls,onecolumn,web]{ieeecolor2}
    \usepackage{generic_one_column}
    \usepackage{amsthm}                                         
    \theoremstyle{definition}                                   
\else
    \ifNonArxivVersion
        \documentclass[journal,twoside,web]{ieeecolor12}              
        \usepackage{generic}                                       
    \else
        \documentclass[journal,twoside,web]{ieeecolor12}              
        \usepackage{generic} 
    \fi
\fi

\usepackage{cite}
\usepackage{amsmath,amssymb,amsfonts}
\usepackage{graphicx}
\usepackage{hyperref}
\hypersetup{hidelinks}
\usepackage{textcomp}
\def\BibTeX{{\rm B\kern-.05em{\sc i\kern-.025em b}\kern-.08em
    T\kern-.1667em\lower.7ex\hbox{E}\kern-.125emX}}
\newcommand{\rev}[1]{{\color{black} #1}}

\usepackage{textcomp}
\let\labelindent\relax
\usepackage{enumitem}   

\newif\ifshowproofs
\showproofstrue

\usepackage{mathtools}
\usepackage{bbm}
\usepackage{graphicx}
\usepackage{empheq}
\usepackage{tikz}
\usepackage{algpseudocode}
\usepackage[noend]{algorithm2e}
\usepackage{subcaption}

\SetKwComment{Comment}{/* }{ */}
\RestyleAlgo{ruled}
\usepackage{pgfplots}
\usepackage[inkscapeformat=png]{svg}

\newtheorem{theorem}{Theorem}[section]
\newtheorem{lemma}[theorem]{Lemma}
\newtheorem{proposition}[theorem]{Proposition}
\newtheorem{remark}[theorem]{Remark}
\newtheorem{corollary}[theorem]{Corollary}
\newtheorem{definition}[theorem]{Definition}

\newtheorem{example}{Example}

\newtheorem{stassumption}{Assumption}

\usepackage{mathtools}
\usepackage{fancyhdr}
\usepackage[T1]{fontenc}
\usepackage{siunitx}
\usepackage{multirow}
\usepackage{placeins}
\usepackage{booktabs}
\usepackage{rotating}
\usepackage{array}
\usepackage{dsfont}
\usepackage{bbold}
\usepackage{adjustbox}
\usetikzlibrary{calc}
\usepackage{xcolor}
\definecolor{niceGreen}{rgb}{0.1, 0.625, 0.1}

\DeclareMathOperator*{\argmin}{argmin}

\pgfplotsset{compat = 1.3}

\pgfplotsset{bsoldot/.style={color=cyan,only marks,mark=*}} \pgfplotsset{bholdot/.style={color=cyan,fill=white,only marks,mark=*}}
\pgfplotsset{rsoldot/.style={color=red,only marks,mark=*}} \pgfplotsset{rholdot/.style={color=red,fill=white,only marks,mark=*}}

\usetikzlibrary{positioning, fit, calc}   
\tikzset{block/.style={draw, thick, text width=0.8\columnwidth ,minimum height=0.5cm, align=center},   
line/.style={-latex}     
}  

\ifsingleCol
    
\fi

\begin{document}
\title{
Computing Optimal Joint Chance Constrained Control Policies}


\author{Niklas Schmid, Marta Fochesato, Sarah H.Q. Li, Tobias Sutter, John Lygeros
\thanks{N.~Schmid, M.~Fochesato, S.H.Q.~Li and J.~Lygeros are with the Automatic Control Laboratory (IfA), ETH Z\"urich, 8092 Z\"urich, Switzerland
        {\tt\small \{nschmid,mfochesato,huilih,jlygeros\}@ethz.ch}}%
\thanks{T.~Sutter is with the Department of Computer Science, University of Konstanz, 78457 Konstanz, Germany
        {\tt\small \ tobias.sutter@uni-konstanz.de}}%
\thanks{Work supported by the European Research Council under grant 787845 (OCAL) and the Swiss National Science Foundation under NCCR Automation under grant 51NF40\_225155.}%
}

\maketitle


\begin{abstract}
We consider the problem of optimally controlling stochastic, Markovian systems subject to joint chance constraints over a finite-time horizon. For such problems, standard Dynamic Programming is inapplicable due to the time correlation of the joint chance constraints, which calls for non-Markovian, and possibly stochastic, policies. Hence, despite the popularity of this problem, solution approaches capable of providing provably-optimal and easy-to-compute policies are still missing. We fill this gap by augmenting the dynamics via a binary state, allowing us to characterize the optimal policies and develop a Dynamic Programming based solution method.
\end{abstract}


\section{Introduction}\label{sec_introduction}
Many 
control applications come with the requirement for safety certificates, e.g., in air-traffic control \cite{Prandini}, self-driving cars \cite{bojarski2016end} or medicine \cite{mannel_1}. 
Safety is commonly defined via a predefined set of safe states, in the sense that any state trajectory leaving the safe set during the control task is considered unsafe. For stochastic systems, safety can only be guaranteed up to some probability. This leads to the definition of so-called joint chance constraints, i.e., constraints that bound the probability to be unsafe with respect to the entire state trajectory. This is in contrast to stage-wise chance constraints, which bound the probability to be unsafe at every individual time-step. The latter constraint type is especially popular for infinite-horizon problems (e.g., in Model Predictive Control) since the probability of remaining safe at all time-steps over an infinite trajectory is zero under mild assumptions on the dynamics \cite{schmid_2,farina_1}. However, for finite-time settings, as considered in this paper, we argue that it is generally more meaningful to define safety as a guarantee on the entire trajectory. 

Despite the ubiquity of joint chance constraints \cite{farina_1,Wang, wang_2,thorpe_1,bavdekar,raghuraman,Ono_2,paulson,ono_5,patil_1},  their practical deployment is hindered by (i) the difficulty in evaluating multivariate integrals \cite{CCDP}; (ii) the non-convex feasible region described by these constraints; and (iii) the time correlation introduced by the constraints, causing a non-Markovian problem structure \cite{farina_1}. While (i) and (ii) equally affect stage-wise and joint chance constraints, (iii) is specific to joint chance constraints. 

To circumvent this time-correlation, approximations like  Boole's inequality have been proposed \cite{Ono_2}. However, such approximation are known to be conservative, especially for long time horizons \cite{patil_1}. More recently, \cite{Wang} suggests augmenting the state vector with a function space to fit the standard Dynamic Programming (DP) format. However, the resulting formulation is computationally challenging, hindering the development of provably-convergent algorithms. 
On a broader scale, joint chance constraints have also been explored for Model Predictive Control, again utilizing Boole's inequality (or similar bounds) \cite{ono_5,paulson,patil_1} or sampling \cite{wang_2,raghuraman,bavdekar, thorpe_1}. Safety constrained MDPs have further been addressed via reinforcement learning, but lack optimality guarantees \cite{chen_1,xu_1,zhang_1,ding_1,tessler_1}. \rev{Finally, infinite-time temporal logic objectives for multiobjective MDPs have been addressed, but their analysis restricts to finite state and action spaces, see \cite{haesaert2021formal, hahn2019interval, etessami2008multi} and references therein. For example, though \cite{haesaert2021formal} discusses continuous state MDPs, analysis and policy computation is limited to finite abstractions.}
 
In summary, to the best of the authors knowledge, the finite time joint chance constrained optimal control problem has not been analysed for continuous state and action spaces. Further, no solutions have been proposed to break the non-Markovian structure imposed by joint chance constraints without resorting to conservative approximations. We aim to fill this gap using DP leading to a policy that achieves an optimal performance under predefined safety constraints. We highlight the following key results of this paper. \rev{
\begin{enumerate}[label=(\roman*)]
\item \textbf{Computation of optimal policies.} We propose a DP recursion for joint chance constraints in continuous-space, finite-time stochastic optimal control problems, optimizing the trade-off between performance and safety.
\item  \textbf{Characterization of optimal policies.} We show that the optimal policy can be constructed as a mixture of two deterministic Markov policies and requires a binary state capturing whether the past trajectory was safe.
\item \textbf{Behaviour of optimal policies.} Our analysis indicates that joint chance-constrained policies can be "controversial" for many applications.
When all safe inputs yield high expected costs the optimal policies might advertise inputs that are likely to violate safety. 
\end{enumerate}

Unlike \cite{Ono_2, ono_5,paulson,raghuraman} our method does not rely on conservative approximations or sampling techniques \cite{bavdekar,wang_2,thorpe_1}, it is computationally superior to \cite{Wang}, and unlike \cite{haesaert2021formal, hahn2019interval, etessami2008multi} we analyse finite time safety under continuous state and action spaces. Our results trivially extend to finite state and input spaces.}


\iffullversion
\smallskip
\textbf{Notation.} We denote by $\mathbb{1}_{\mathcal{A}}(x)$ the indicator function of a set $\mathcal{A}$, where $\mathbb{1}_{\mathcal{A}}(x)=1$ if $x\in {\mathcal{A}}$ and $\mathbb{1}_{\mathcal{A}}(x)=0$ otherwise. Given two sets $\mathcal{X}, \mathcal{Y}$, the difference between them is denoted by $\mathcal{X}\setminus \mathcal{Y} = \{x\in \mathcal{X}: x\notin \mathcal{Y}\}$, while the complement of a set $\mathcal{X}$ is denoted as $\mathcal{X}^c$. 
We denote by $[N]$ the set $\{0,1,\dots,N\}$ and $\mathbb{R}_{\geq0}$ the non-negative reals. Further, $\land$ and $\lor$ symbolize the logical conjunction and disjunction, respectively.
\fi

\section{Preliminaries and Problem Formulation}\label{sec_problem_formulation}
\textbf{Markov Decision Processes.} We define a safety-constrained discrete-time stochastic system over a finite time-horizon as a tuple $(\mathcal{X}, \mathcal{U}, T, \ell_{0:N}, \mathcal{A})$, where the state space $\mathcal{X}$ and the input space $\mathcal{U}$ are Borel subsets of complete separable metric spaces equipped respectively with $\sigma$-algebras $\mathcal{B}(\mathcal{X})$ and $\mathcal{B}(\mathcal{U})$. 
Given a state $x_k\in\mathcal{X}$ and an input $u_k\in\mathcal{U}$, the Borel-measurable stochastic kernel $T:\mathcal{B}(\mathcal{X})\times\mathcal{X}\times\mathcal{U}\rightarrow[0,1]$ describes the stochastic state evolution, leading to $x_{k+1}\sim T(\cdot|x_k,u_k)$.
Additionally, $\ell_k: \mathcal{X}\times\mathcal{U} \rightarrow \mathbb{R}_{\geq 0}, k\in[N-1]$ and $\ell_N: \mathcal{X} \rightarrow \mathbb{R}_{\geq 0}$ denote measurable, non-negative functions called stage and terminal cost, respectively, which are incurred at every time-step $k\in[N-1]$ and at terminal time $N\in\mathbb{N}$. Finally, let the set $\mathcal{A} \subseteq \mathcal{B}(\mathcal{X})$ denote a safe set. We define a trajectory to be safe, if $x_{0:N}\in\mathcal{A}$. By abuse of notation, we interpret $x_{0:N}\in \mathcal{A}$ to mean $x_k\in \mathcal{A}$ for all $k\in [N]$. 



For $k \in[N-1]$ we define the space of histories up to time $k$ recursively as $\mathcal{H}_k = \mathcal{X}\times\mathcal{U} \times \mathcal{H}_{k-1}$, with $\mathcal{H}_0 = \mathcal{X}$; a generic element $h_k \in {\mathcal{H}_k}$ is of the form $h_k = (x_0, u_0, x_1, u_1, \ldots, x_{k-1}, u_{k-1}, x_k)$. 
 


A stochastic policy is a sequence $\pi=(\mu_0,\dots,\mu_{N-1})$ of Borel-measurable stochastic kernels $\mu_k$, $k\in[N-1]$, that, given $h_k$, assigns a probability measure $\mu_k(\cdot,h_k)$ on the set $\mathcal{B}(\mathcal{U})$. A deterministic policy is the special case $\pi_{\text{d}}=(\mu_0,\dots,\mu_{N-1})$ where $\mu_k: \mathcal{H}_k\rightarrow\mathcal{U}$, $k\in [N-1]$ are simply measurable maps. A policy is further called Markov if $\mu_k(\cdot | h_k) = \mu_k(\cdot | x_k)$ for all $h_k\in \mathcal{H}_k$ and $k\in[N-1]$. Otherwise, it is called causal. We denote the set of stochastic causal, stochastic Markov, deterministic causal and deterministic Markov policies by $\Pi$, $\Pi_{\text{m}}$, $\Pi_{\text{d}}$ and $\Pi_{\text{dm}}$, respectively, and note that $\Pi_{\text{dm}}\subseteq \Pi_{\text{m}}\subseteq \Pi$, and $\Pi_{\text{dm}}\subseteq \Pi_{\text{d}}\subseteq \Pi$.
A mixed policy $\pi_{\text{mix}}$ describes a deterministic policy that is randomly chosen at the initial time-step $k=0$ and then used during the entire control task. \rev{In other words, before applying any control inputs to the system, a deterministic Markov policy is sampled according to some probability law $\pi\sim\mathbb{P}(\Pi_{\text{d}})$, and then this policy $\pi\in\Pi_{\text{d}}$ is used to compute control inputs for all time-steps $k\in[N-1]$. Note that this is in contrast to stochastic policies, where an action is sampled at every time-step from a history dependent distribution.}
More formally, a mixed policy can be defined as in \cite{fainberg1982non}. We endow the set of deterministic policies $\Pi_{\text{d}}$ with a metric topology and define the corresponding Borel $\sigma$-algebra by $\mathcal{G}_{\Pi_{\text{d}}}$. A mixed policy $\pi_{\text{mix}}$ can be seen as a random variable to the measurable space $(\Pi_{\text{d}},\mathcal{G}_{\Pi_{\text{d}}})$. The set of all mixed policies is denoted as $\Pi_{\text{mix}}$. 

\rev{\textbf{Dynamic Programming.} We consider the following standing assumption. }
\begin{stassumption}{}\label{assumption_of_attainability}
    The input set $\mathcal{U}$ is compact. Furthermore, for every $x\in\mathcal{X}$, $k\in [N-1]$ and $\mathcal{A}\in\mathcal{B}(\mathcal{X})$, the transition kernel $T(\mathcal{A}|x,\cdot)$ and stage cost $\ell_k(x,\cdot)$ are continuous.
\end{stassumption}

For a given initial state $x_0\in\mathcal{X}$, policy $\pi\in\Pi$, and transition kernel $T$, a unique probability measure $\mathbb{P}_{x_0}^{\pi}$ for the state-input-trajectory is defined over $\mathcal{B}((\mathcal{X} \times \mathcal{U})^N\times \mathcal{X})$, which can be sampled recursively via $x_{k+1}\!\sim\!T(\cdot|x_k,u_k)$ with $u_k\!\sim\!\mu_k(h_k)$ \cite{Abate_1}. We define the expected cost as  
\begin{align*}
    &\mathbb{E}_{x_0}^\pi 
\left[\ell_N(x_N)+\sum_{k=0}^{N-1}\ell_k(x_k,u_k) \right] =\int\limits_{(\mathcal{X} \times \mathcal{U})^N\times \mathcal{X}}\biggl(\ell_N(x_N)+\\ & \qquad\sum_{k=0}^{N-1} \ell_k(x_k,u_k)\biggl)\mathbb{P}_{x_0}^{\pi}(dx_0,du_0,\dots,dx_N).
\end{align*} Let $C_k^{\star}(x_k)=\allowbreak\inf_{\pi \in \Pi_{\text{dm}}}\mathbb{E}_{x_k}^\pi 
\left[ \ell_N(x_N) + \sum_{j=k}^{N-1} \ell_j(x_j,u_j) \right]$. The finite-time optimal control problem computes $C_0^{\star}(x_0)$, where \cite{Bertsekas_1}
\begin{align}
        C_N^{\star}(x_N)\!&=\!\ell_N(x_N), \label{eq_max_cost_evaluation} \\
         C_k^{\star}(x_k)\!&=\hspace{-3pt}\inf_{u_k\in\mathcal{U}}\!\ell_k(x_k,u_k)\!+\!\int_{\mathcal{X}}\hspace{-3pt}C_{k+1}^{\star}(x_{k+1})T(dx_{k+1}|x_k,u_k). \nonumber
\end{align}
The infimum is attained under Assumption \ref{assumption_of_attainability} \cite{Bertsekas_1} and the resulting deterministic Markov policy is also optimal within the class of stochastic causal policies \cite[Theorem~3.2.1]{Herandez_1}.

The probability of safety is defined as
 $   \mathbb{P}_{x_0}^\pi(x_{0:N}\in \mathcal{A})=\int\displaylimits_{(\mathcal{X} \times \mathcal{U})^N\times \mathcal{X}}\!\prod_{k=0}^{N-1}\! \mathbb{1}_{\mathcal{A}}(x_k)\mathbb{P}_{x_0}^{\pi}(dx_0,du_0,\dots,dx_N).$
We use $V_k^{\star}:\mathcal{X}\rightarrow[0,1]$ as shorthand notation for $V_k^{\star}(x_k)=\sup_{\pi \in \Pi_{\text{dm}}} \mathbb{P}(x_{k:N}\in \mathcal{A}|x_k, \pi)$. Then, following \cite{Abate_1},
\begin{align}
        V_N^{\star}(x_N) &=\mathbb{1}_{\mathcal{A}}(x_N), \label{eq_max_invariance_evaluation}\\
        V_k^{\star}(x_k) &=\sup_{u_k\in\mathcal{U}} \mathbb{1}_\mathcal{A}(x_k)\int_{\mathcal{X}}V_{k+1}^{\star}(x_{k+1})T(dx_{k+1}|x_{k}, u_k). \nonumber
\end{align} where the supremum is attained under Assumption \ref{assumption_of_attainability} \cite{schmid_1}. \rev{The recursions \eqref{eq_max_cost_evaluation} and \eqref{eq_max_invariance_evaluation} can equivalently be used to evaluate the cost and safety of a given policy $\pi$ by restricting $\mathcal{U}=\{\pi(x_k)\}$ for every $x_k\in\mathcal{X}$. We then denote the associated value functions by $C_k^{\pi}$ and $V_k^{\pi}$, $k\in[N]$, respectively.} 

\textbf{Problem Formulation.} Our aim is to minimize the cumulative cost, while guaranteeing a prescribed level of safety over the entire time horizon, expressed as
\begin{equation}\label{eq_problemFormulation_CCOC}
\begin{aligned} 
    \begin{split}
    \inf_{\pi \in \Pi} \quad  &\mathbb{E}_{x_0}^\pi\left[ \ell_N(x_N) + \sum_{k=0}^{N-1} \ell_k(x_k,u_k) \right]\\
    \text{subject to} \quad  & \mathbb{P}_{x_0}^\pi(x_{0:N}\in \mathcal{A})\geq \alpha,
    \end{split}
\end{aligned}
\end{equation}
where $\alpha \in \mathbb{R}_{\geq 0}$ is a user-specified risk tolerance parameter. The constraint realizes the requirement for the state trajectory $x_{0:N}$ to lie in the safe set $\mathcal{A}$ with a probability of at least $\alpha$. 
\section{Joint Chance Constr. Dynamic Programming}
\label{sec_Extact_JCOCC}

\rev{Thanks to the equivalence of stochastic causal and deterministic mixed policies, we solve Problem \eqref{eq_problemFormulation_CCOC} over the policies $\pi_{\text{mix}}\in\Pi_{\text{mix}}$ \cite{fainberg1982non}. 
}

\subsection{Lagrangian Dual Framework}
\label{sec_lagrange_dual_framework}
In line with \cite{Ono_2}, we rely on the Lagrangian dual formulation of Problem \eqref{eq_problemFormulation_CCOC}, which, thanks to the equivalence of mixed and stochastic causal policies \cite{fainberg1982non}, is given by 
\begin{equation*}\label{eq_lagrange_dual_over_stochasticCausal_0}
    \begin{split}
        \max_{\lambda \in \mathbb{R}_{\geq0}} \inf_{\pi \in \Pi_{\text{mix}}}\! &\mathbb{E}_{x_0}^\pi\!\left[ \ell_N(x_N)\!+\!\sum_{k=0}^{N-1}\!\ell_k(x_k,u_k)\right]
         \\&+ \lambda(\alpha\!-\!\mathbb{P}_{x_0}^\pi(x_{0:N}\in \mathcal{A})). \\
     \end{split}
\end{equation*}
Note that the inner infimum can be equivalently cast as
\begin{align}
\begin{split}
        \inf_{\pi \in \Pi_{\text{mix}}}\! &\mathbb{E}_{x_0}^\pi\!\left[ \ell_N(x_N)\!+\!\sum_{k=0}^{N-1}\!\ell_k(x_k,u_k)
     \! -\!\lambda\!\prod_{k=0}^N\!\mathbb{1}_{\mathcal{A}}(x_k)\right]  +\lambda \alpha.
\end{split}
     \label{eq_innerproblem_intro}
\end{align}

Unfortunately, this dual problem is notoriously hard to solve, even in the conceptually much simpler space of deterministic Markov policies. To see this, one might naively try to apply DP by defining the terminal and stage costs as
\begin{equation}
    \begin{split}
    \ell_{\lambda,N}(x_N) &=\ell_N(x_N) -\lambda \prod_{k=0}^N\mathbb{1}_{\mathcal{A}}(x_k) + \lambda \alpha,\\
    \ell_{\lambda,k}(x_k,u_k) &= \ell_k(x_k,u_k).
\end{split}\label{eq_dp_recursion_naive_productsum}
\end{equation}
The terminal cost depends on the full trajectory; information we only have at time-step $N$, but whose distribution depends on the policy at time-steps $k=0,\dots,N-1$. This non-Markov structure prevents using DP. \rev{In the following, we recover the Markovian property via a state augmentation, reformulate the Lagrange dual as a bilevel problem, show strong duality to Problem \eqref{eq_problemFormulation_CCOC}, then present an algorithm to solve this bilevel problem and, by strong duality, also Problem \eqref{eq_problemFormulation_CCOC} (Fig.~\ref{fig_paper_structure}). }
\ifsingleCol
    \begin{figure}
        \centering
        \includegraphics[height=5cm]{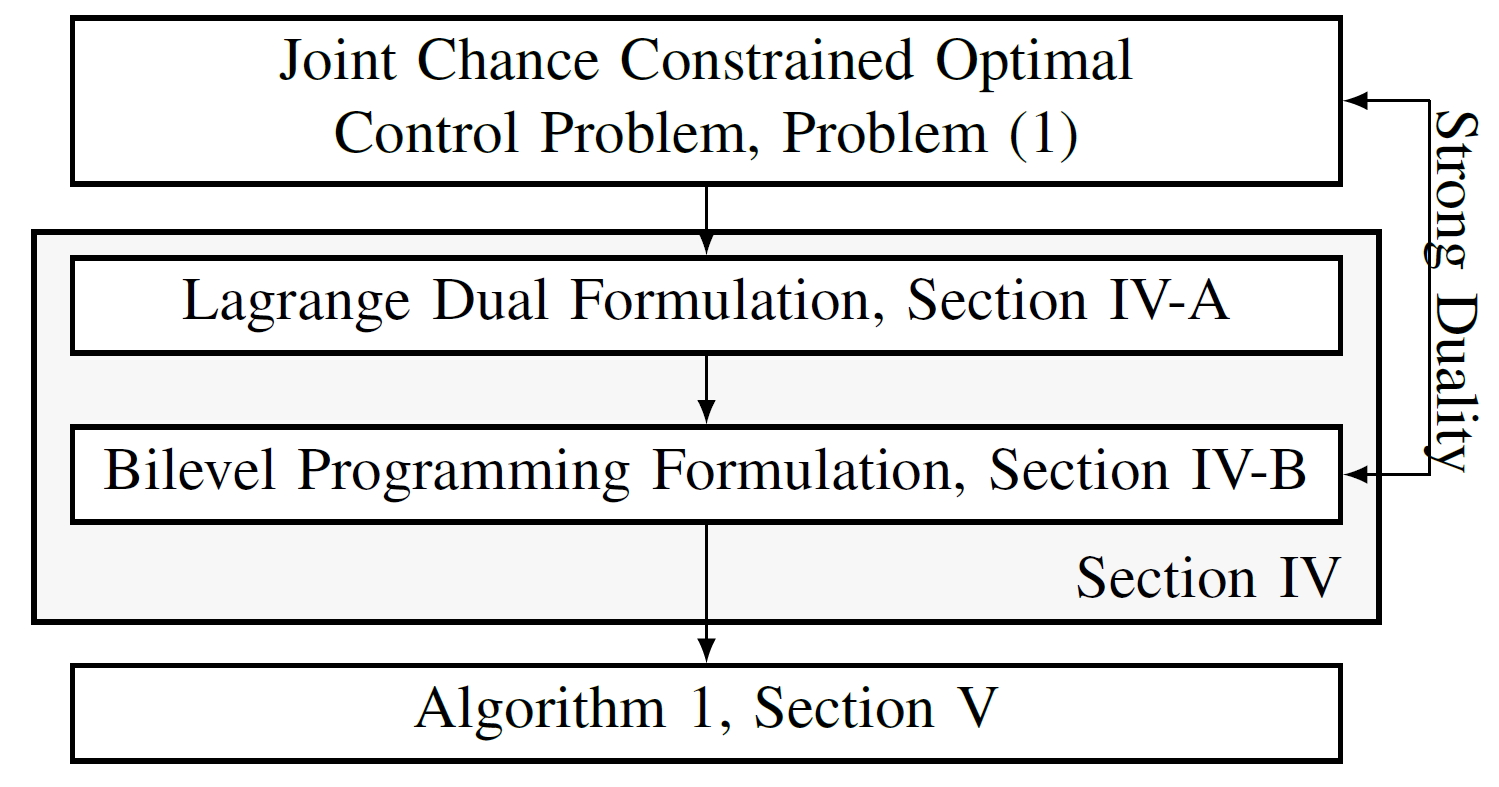}
        \caption{Graphical representation of the paper structure.}
        \label{fig_paper_structure}
    \end{figure}
\else
    \begin{figure}
        \centering
        \begin{tikzpicture}  
        \node[block] (a) {Joint Chance Constrained Problem \eqref{eq_problemFormulation_CCOC}};  
        \node[block,below=0.4 cm of a, fill=white] (c) {Lagrange Dual Formulation \eqref{eq_lambda_optimal_policy_eq}};  
        \node[block,below=0.4 cm of c, fill=white] (d) {Bilevel Programming Formulation \eqref{bilevel}}; 
        \draw[line] (a)-> (c);  
        \draw[line] (c.south)++(0,-0.15)-> (d);  
        \draw[line] (d.north)++(0,0.15) -> (c);  
        \draw [line] (a.east) -- (a.east) ++(.5,0) |- (d.east);
        \draw [line] (a.east)++(.5,0) -- (a.east) ;    
        \draw (4.3,-1.0) node[rotate=-90] {Strong Duality} ;    
        \draw (1.2,-0.5) node[] {Lagrange Dual};
        \draw (1.025,-1.5) node[] {Equivalence};
        \end{tikzpicture} 
        \caption{Graphical representation of the paper structure.}
        \label{fig_paper_structure}
    \end{figure}
\fi

To overcome the non-Markovian structure of Problem \eqref{eq_innerproblem_intro}, we start with an intuitive observation: The product $\prod_{k=0}^N\mathbb{1}_{\mathcal{A}}(x_k)$ in the terminal cost function \eqref{eq_dp_recursion_naive_productsum} depends on the state trajectory up to time $N$, but in the end it simply equals one if all states remained safe, and zero otherwise. We can easily encode this information by introducing a binary state $b_k$ that, inspired by recursion \eqref{eq_max_invariance_evaluation}, is initialized to $\mathbb{1}_{\mathcal{A}}(x_0)$, and is set to zero whenever the trajectory leaves the safe set. Consequently, we can formulate the product in \eqref{eq_dp_recursion_naive_productsum} in terms of the value of the binary state at terminal time and recover the Markovian structure of the problem. More formally, we define as $
\Tilde{\mathcal{X}}= \mathcal{X}\times\{0,1\}$ and
    $\Tilde{x}_k = (x_k, b_k)\in\Tilde{\mathcal{X}}$, with $b_0=\mathbb{1}_{\mathcal{A}}(x_0)$ and
    $b_{k+1} = \mathbb{1}_{\mathcal{A}}(x_{k+1})b_{k}$, leading to the overall dynamics $\Tilde{T}:\mathcal{B}(\Tilde{\mathcal{X}})\times\Tilde{\mathcal{X}}\times\mathcal{U}\rightarrow[0,1]$,
\iffullversion
\begin{align*}
    &\Tilde{T}(\Tilde{x}_{k+1}|\Tilde{x}_k,u_k) \\&\!=\!\begin{cases}
    T(x_{k+1}|x_k,\!u_k)\!& \text{if} \ (b_{k}\!=\!0 \land b_{k+1}\!=\!0) \\& \ \ \ \lor \  (b_{k}=1 \land b_{k+1}\!=\!\mathbb{1}_{\mathcal{A}}(x_{k+1})) \\
    0 & \text{otherwise},
\end{cases}
\end{align*}
or equivalently 
\fi
\begin{align*}
    \Tilde{T}(\Tilde{x}_{k+1}|\Tilde{x}_k,u_k) &= 
    T(x_{k+1}|x_k,u_k) \big((1-b_{k})(1-b_{k+1}) \\& \quad + b_k(1-|b_{k+1}-\mathbb{1}_{\mathcal{A}} (x_{k+1})|)\big),
\end{align*}
which is again measurable.
For readability, we overload the notation $\Tilde{T}(\Tilde{x}_{k+1}|\Tilde{x}_k,u_k)=T(\Tilde{x}_{k+1}|\Tilde{x}_k,u_k)$, $\mathbb{1}_{\mathcal{A}}(\Tilde{x}_k)=\mathbb{1}_{\mathcal{A}}(x_k)$ and $\ell_N(\tilde{x}_N)=\ell_N(x_N), \ell_k(\tilde{x}_k,u_k)=\ell_k(x_k,u_k)$. Policies are defined as before replacing $\mathcal{X}$ with $\Tilde{\mathcal{X}}$. 

The expected product of the set indicator functions in \eqref{eq_innerproblem_intro} is now given by the expected value of the binary state,
$\mathbb{E}_{\tilde{x}_{0}}^\pi\left[\prod_{k=0}^N\mathbb{1}_{\mathcal{A}}(\Tilde{x}_{k}) \right]=\mathbb{E}_{\tilde{x}_{0}}^\pi\left[ b_N \right]$, and \eqref{eq_innerproblem_intro} becomes
\begin{equation}\label{eq_lagrange_dual_over_stochasticCausal}
    \begin{split}
        \max_{\lambda \in \mathbb{R}_{\geq0}} \inf_{\pi \in \Pi_{\text{mix}}}\hspace{-7pt}\mathbb{E}_{\tilde{x}_{0}}^\pi\!\Big[ \ell_N(\tilde{x}_N)\!+\hspace{-4pt}\sum_{k=0}^{N-1}\!\ell_k(\tilde{x}_k,u_k)
         \!+\!\lambda(\alpha\!-\!b_N)\Big]. 
     \end{split}
\end{equation}
The inner minimization problem is now an MDP whose optimal solution for a fixed $\lambda$ is attained by deterministic Markov policies.
\rev{This state-augmentation trick can be interpreted as constructing a product MDP with an automaton, which describes the safety specification and is similarly done in \cite{etessami2008multi, hahn2019interval, haesaert2021formal} for temporal logic specifications.}



\subsection{Bilevel Framework}
\label{sec_bilevel_programming_approach}
To simplify notation, note that using the equivalence $\mathbb{E}_{\tilde{x}_{0}}^{\pi}\left[ b_N \right] = V_0^\pi(\tilde{x}_0)$, Problem \eqref{eq_lagrange_dual_over_stochasticCausal} reduces to
\iffullversion
\begin{align}
    \begin{split}
        &\max_{\lambda \in \mathbb{R}_{\geq0}}\inf_{\pi \in \Pi_{\text{mix}}} \hspace{-7pt}\mathbb{E}_{\tilde{x}_{0}}^\pi\left[ \ell_N(\tilde{x}_N)\!+\hspace{-4pt}\sum_{k=0}^{N-1} \ell_k(\tilde{x}_k,u_k)
         -\lambda b_N\right]  + \lambda \alpha \\
         =\ &\max_{\lambda \in \mathbb{R}_{\geq0}}\inf_{\pi \in \Pi_{\text{mix}}} C_0^{\pi}(\tilde{x}_0) - \lambda V_0^{\pi}(\tilde{x}_0) + \lambda \alpha.
    \end{split}
    \label{eq_lambda_optimal_policy_eq}
\end{align}
\fi
Intuitively, this means that the inner Problem reduces to finding a policy that minimizes the expected cost while being rewarded by $-\lambda$ for remaining safe. 

\iffullversion 
Referring to Appendix \ref{sec_duality} on duality, we note the following: %
\fi
We have the constraint $g(\pi)= \alpha - V_0^{\pi}(\tilde{x}_0)$ and Lagrangian 
         $q(\lambda)=\inf_{\pi \in \Pi_{\text{mix}}} C_0^{\pi}(\tilde{x}_0) - \lambda V_0^{\pi}(\tilde{x}_0)+ \lambda \alpha,$
which is a concave function in $\lambda$ and where we will show later that the infimum is attained by some $\pi_{\lambda}\in\Pi_{\text{mix}}$. Then $\alpha - V_0^{\pi_{\lambda}}\in\partial q(\lambda)$. Hence, if $V_0^{\pi_{\lambda}}<\alpha$ then $\lambda\leq\lambda^{\star}$ and $\lambda\geq\lambda^{\star}$ otherwise. 
\iffullversion
A graphical illustration is given in Fig.~\ref{fig_subgradients_dual}.
\ifsingleCol
    \begin{figure}[!tb]
    \centering
        \includegraphics[height=3.5cm]{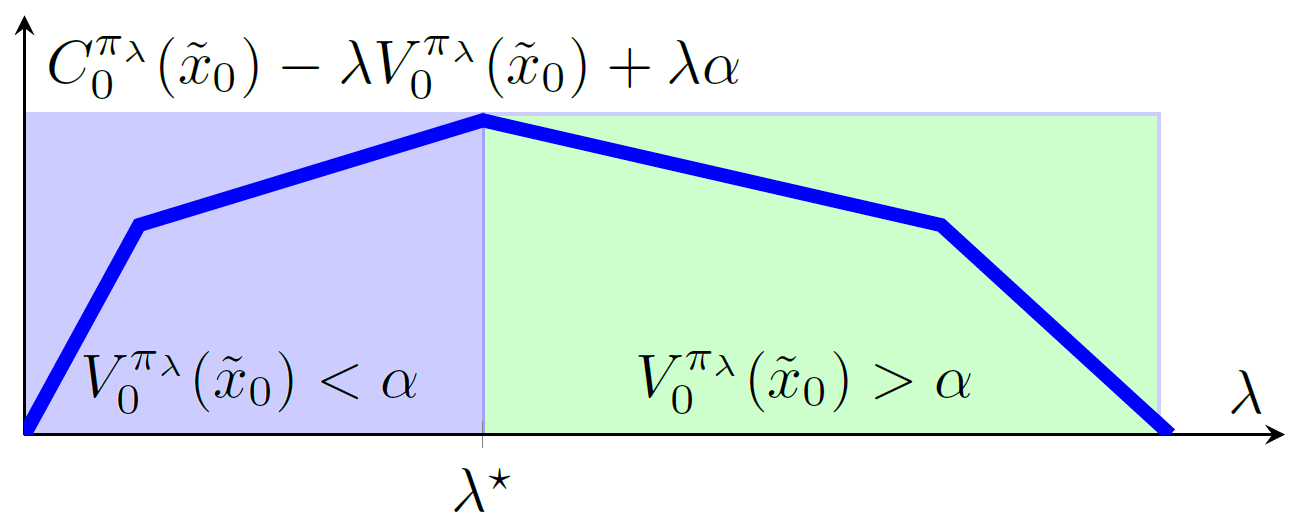}
        \caption{An illustration of a objective function for the outer maximization in Problem \eqref{eq_lagrange_dual_over_stochasticCausal} (inspired by \cite{Ono_2}). The policy $\pi_{\lambda}$ denotes an optimal argument to the inner minimization of Problem \eqref{eq_lagrange_dual_over_stochasticCausal} under given $\lambda$ (assuming it exists) and $C_0^{\pi_{\lambda}}(\Tilde{x}_0),V_0^{\pi_{\lambda}}(\Tilde{x}_0)$ the cost and safety associated with that policy.} 
        \label{fig_subgradients_dual} 
    \end{figure}
\else
    \begin{figure}[!tb]
    \centering
        \begin{tikzpicture}
            \filldraw[draw=blue, fill=green, opacity=0.2] (2.65,0) rectangle ++(3.9,1.85);
            \filldraw[draw=blue, fill=blue, opacity=0.2] (0,0) rectangle ++(2.65,1.85);
            \begin{axis}[width=\columnwidth,height=4cm, xlabel={$\lambda$}, ylabel={$C^{\pi_{\lambda}}_0(\Tilde{x}_0)-\lambda V^{\pi_{\lambda}}_0(\Tilde{x}_0)+\lambda\alpha$}, xtick={0,2}, ytick={0,1}, axis x line = bottom,axis y line = left, axis lines=middle,xticklabels={0,$\lambda^{\star}$},ymax=3, xmax=5.5]
            \addplot +[line width=0.8mm, mark=none] coordinates {(0,1) (.5,2) (2,2.5) (4,2) (5,1)};
            \end{axis};
            \draw (1.3,0.3) node {$V_0^{\pi_{\lambda}}(\Tilde{x}_0)<\alpha$};
            \draw (4.5,0.3) node {$V_0^{\pi_{\lambda}}(\Tilde{x}_0)>\alpha$};
        \end{tikzpicture}
        \caption{An illustration of a objective function for the outer maximization in Problem \eqref{eq_lagrange_dual_over_stochasticCausal} (inspired by \cite{Ono_2}). The policy $\pi_{\lambda}$ denotes an optimal argument to the inner minimization of Problem \eqref{eq_lagrange_dual_over_stochasticCausal} under given $\lambda$ (assuming it exists) and $C_0^{\pi_{\lambda}}(\Tilde{x}_0),V_0^{\pi_{\lambda}}(\Tilde{x}_0)$ the cost and safety associated with that policy.} 
        \label{fig_subgradients_dual} 
    \end{figure}
\fi
\fi
\rev{In other words, choosing $\lambda$ larger than some maximizing $\lambda^{\star}$ for $q(\cdot)$ yields the safety probability of the policy ${\pi_{\lambda}}$ at least $\alpha$; choosing $\lambda$ smaller than $\lambda^{\star}$ yields the safety probability at most $\alpha$. Hence, to find $\lambda^{\star}$, we can search for the smallest $\lambda$ which results in a policy ${\pi_{\lambda}}$ with a safety of at least $\alpha$. Problem \eqref{eq_lambda_optimal_policy_eq} can thus be formulated as the bilevel program}
\begin{equation}\label{bilevel}
    \begin{split}
        \min_{\lambda\in\mathbb{R}_{\geq 0}} & \ \lambda\\
        \text{subject to} & \quad \pi_{\lambda} \in {\argmin}_{\pi\in\Pi_{\text{mix}}}  C_0^{\pi}(\Tilde{x}_0) - \lambda V_0^{\pi}(\Tilde{x}_0) \\
        &\quad   V_0^{\pi_{\lambda}}(\Tilde{x}_0)\geq \alpha,
    \end{split}
\end{equation}
where we drop the term $\lambda \alpha$ in the first constraint of Problem \eqref{eq_lambda_optimal_policy_eq} since it is a constant and does not affect the optimal argument $\pi_{\lambda}$. \rev{
The structure of \eqref{bilevel} motivates our algorithm: we first compute $\pi_{\lambda}$ and then use bisection to find the lowest $\lambda$ that generates a policy $\pi_{\lambda}$ with a safety of at least $\alpha$.
We show strong duality between \eqref{eq_problemFormulation_CCOC} and \eqref{eq_lambda_optimal_policy_eq} via the equivalence of \eqref{eq_problemFormulation_CCOC} and \eqref{bilevel} at the end of this section.
}

 
\subsubsection{Attainability of Optimal Deterministic Policies}
We first restrict ourselves to deterministic policies and solve
\begin{align}
    \inf_{\pi \in \Pi_{\text{d}}} C_0^{\pi}(\tilde{x}_0) - \lambda V_0^{\pi}(\tilde{x}_0).
    \label{eq_inner_problem_definedViaCV}
\end{align}
Since the problem is an MDP, it is optimally solved by a deterministic Markov policy \cite[Theorem~3.2.1]{Herandez_1}, which can be computed via the DP recursion 
\begin{align}
    J_N^{\star}(\Tilde{x}_N)\!&=\!\ell_N(\Tilde{x}_N)\!-\!\lambda b_N,\label{eq_our_dp_recursion_innerdual} \\
    J_k^{\star}(\Tilde{x}_k)\!&=\hspace{-.5em}\inf_{u_k\in\mathcal{U}} \hspace{-.2em}\ell_k(\Tilde{x}_k,u_k)\!+\!\int_{\Tilde{\mathcal{X}}} \hspace{-.5em}J_{k+1}^{\star}(\Tilde{x}_{k+1})\tilde{T}(d\Tilde{x}_{k+1}|\Tilde{x}_k,u_k). \nonumber
\end{align}

\iffullversion
\begin{theorem}\label{th_recursion}(Attainability of Deterministic Markov Policies) 
\else
\begin{theorem}\label{th_recursion}(Attainability of Deterministic Markov Policies) 
\fi
Given a fixed $\lambda \in \mathbb{R}_{\geq 0}$, there exists a measurable deterministic Markov policy that attains the infimum in \eqref{eq_our_dp_recursion_innerdual} at every time-step $k\in[N]$ and is also an optimal solution to Problem \eqref{eq_inner_problem_definedViaCV}. The resulting $J_k^{\star}$ is measurable for all $k \in [N]$. 
\label{thm_our_Dp-recursion}
\end{theorem}
    \begin{proof}
        \rev{Let $g_k(\Tilde{x}_k,u_k)=\ell_k(\Tilde{x}_k,u_k)$, $g_N(\Tilde{x}_N)=\ell_N(\Tilde{x}_N)-\lambda b_N$, which are measurable and continuous in $u_k$ for all $\Tilde{x}_k\in\Tilde{\mathcal{X}}$ by Assumption 1. Then,
        \begin{align*}
            \inf_{\pi \in \Pi_{\text{d}}}\hspace{-4pt}C_0^{\pi}(\tilde{x}_0)\!-\!\lambda V_0^{\pi}(\tilde{x}_0)\!=\hspace{-4pt}\inf_{\pi \in \Pi_{\text{d}}}\hspace{-2pt}\mathbb{E}_{\tilde{x}_{0}}^\pi\!\left[ g_N(\tilde{x}_N)\!+\hspace{-4pt}\sum_{k=0}^{N-1}\!g_k(\tilde{x}_k,u_k)\right]\hspace{-2pt},
        \end{align*}
        which is solved by the DP recursion \cite{Bertsekas_1}
        \begin{align*}
            J_N^{\star}(\Tilde{x}_N) 
            &=  g_N(\Tilde{x}_N) = l_N(\Tilde{x}_N) - \lambda b_N\\
            J_k^{\star}(\Tilde{x}_k) &=\inf_{u_k\in\mathcal{U}}g_k(\Tilde{x}_k, u_k)\!+\hspace{-3pt}\int_{\Tilde{\mathcal{X}}}\hspace{-4pt}J_{k+1}^{\star}(\Tilde{x}_{k+1})\Tilde{T}(d\Tilde{x}_{k+1}|\Tilde{x}_k,u_k).
        \end{align*}
        }
    \iffullversion
       To complete the proof, we follow the induction argument in \cite[Prop.~1]{Kariotoglou_1}. Assume $J_{k+1}^{\star}(\cdot)$ is measurable and let
        $F(\Tilde{x}_k,u_k)=g_k(\Tilde{x}_k,u_k) + \int_{\Tilde{\mathcal{X}}}J_{k+1}^{\star}(\Tilde{x}_{k+1})\tilde{T}(d\Tilde{x}_{k+1}|\Tilde{x}_k,u_k)$. Since $g_k(\Tilde{x}_k,u_k)$ and $\tilde{T}(\Tilde{x}_{k+1}|x_k,u_k)$ are continuous in $u_k$, $F(\Tilde{x}_k,u_k)$ is continuous in $u_k$ for every $\Tilde{x}_k\in\Tilde{\mathcal{X}}$ \cite[Fact~3.9]{Nowak_1}. Since $\mathcal{U}$ is compact, the supremum is thus attained by a measurable map $\mu^{\star}_k(\cdot)$ \cite[Corollary~1]{Brown_1}. 
        Using the fact that $F(\cdot,\cdot)$ and $\mu^{\star}_k(\cdot)$ are measurable, we have that $F(\cdot,\mu^{\star}_k(\cdot))$ is measurable; since 
        $g_k(\cdot,\cdot)$ is measurable,  $J_k^{\star}(\cdot)$ is measurable \cite[Prop.~7.29]{Bertsekas_1}. As $J_N^{\star}(\cdot)=g_N(\cdot)$ is measurable, it follows by induction that $J_k^{\star}(\cdot)$ is measurable and the infimum is attained by a measurable map $\mu^{\star}_k(\cdot)$ for all $k \in [N-1]$.
    \fi
    \end{proof} 
\subsubsection{Attainability of Optimal Mixed Policies}

Throughout the remainder of the paper, we denote by $\overline{C}(\Tilde{x}_0),\overline{V}(\Tilde{x}_0)$ the control cost and safety associated with the maximum safety recursion \eqref{eq_max_invariance_evaluation} and $\underline{C}(\Tilde{x}_0),\underline{V}(\Tilde{x}_0)$ the control cost and safety associated with the minimum cost recursion \eqref{eq_max_cost_evaluation}. 
\begin{definition}
For a given initial condition $\Tilde{x}_0$, $\overline{C}(\Tilde{x}_0)$ bounded, we define the Performance Set attainable under any policy class $\Gamma$ as 
\begin{align*}
    P_{\Gamma,\Tilde{x}_0}=\bigcup_{\pi\in\Gamma}\left\{\left(C_0^{\pi}(\Tilde{x}_0),V_0^{\pi}(\Tilde{x}_0)\right)\right\}
\end{align*}
and the Pareto front $P^{\star}_{\Gamma,\Tilde{x}_0}\!:\![\underline{V}(\Tilde{x}_0),\overline{V}(\Tilde{x}_0)]\!\rightarrow\![\underline{C}(\Tilde{x}_0),\overline{C}(\Tilde{x}_0)]$, 
\begin{equation*}
\begin{array}{lcl}
P_{\Gamma,\Tilde{x}_0}^{\star}(p) = &  \underset{C,V}{\inf} &  C\\ 
        &  \text{subject to} &  (C,V) \in P_{\Gamma,\Tilde{x}_0}\\
         & &  V \geq p.
\end{array}
\end{equation*}
Further, given $\lambda\in\mathbb{R}$, we denote as $\lambda$-optimal under the policy class $\Gamma$ any policy $\pi_{\lambda}\in\Gamma_{\lambda}$ with $\Gamma_{\lambda}= \{\pi'\in\Gamma: \pi'\in{\argmin}_{\pi\in\Gamma}  C_0^{\pi}(\Tilde{x}_0) - \lambda V_0^{\pi}(\Tilde{x}_0)\}$.
\end{definition}
\smallskip
The Pareto front comprises the points in $P_{\Gamma,\Tilde{x}_0}$, for which no cost reduction can be achieved without sacrificing safety and vice versa. Thus, it is monotone non-decreasing by definition (see also Fig.~\ref{fig_pareto_front_definition} for a graphical illustration). 
\smallskip
\begin{proposition}
    \label{thm_convexity_of_mixed_performance_sets} (Reduction to Mixed Policies)
    The Performance Set associated with mixed policies $P_{\Pi_{\text{mix}},\Tilde{x}_0}$ is the convex hull of the Performance Set $P_{\Pi_{\text{d}},\Tilde{x}_0}$ of deterministic causal policies.  
\end{proposition}
\iffullversion
\begin{proof}
    \rev{Similar results have been reported for finite state and action MDPs \cite{altman2021constrained}. For completeness, we provide a proof for continuous state and action MDPs.}

    We first show that the performance of any mixed policy is contained in the convex hull of the Performance Set $\text{Conv}(P_{\text{d},\Tilde{x}_0})$ of deterministic causal policies.
    Recall that a mixed policy is defined as a mixture of deterministic causal policies. Let $\pi_1,\dots,\pi_M\in\Pi_{\text{d}}$, which are sampled with probabilities $\eta_1,\dots,\eta_M\in\mathbb{R}_{\geq 0}$, where $\sum_{i=0}^M\eta_i=1$. Let $(C_1,V_1),\dots,(C_M,V_M)\in P_{\text{d},\Tilde{x}_0}$ be the pairs of control cost and safety associated with $\pi_1,\dots,\pi_M$. Then the corresponding mixed policy has safety $V=\sum_{i=0}^M \eta_iV_i$ and control cost $C=\sum_{i=0}^M \eta_iC_i$. Hence, $(C,V)\in \text{Conv}(P_{\text{d},\Tilde{x}_0})$.

    We next show that any point in the convex hull $\text{Conv}(P_{\text{d},\Tilde{x}_0})$ of the Performance Set of $\Pi_{\text{d}}$ is associated with a mixed policy. Let $(C,V)\in \text{Conv}(P_{\text{d},\Tilde{x}_0})$. \rev{Recall that $P_{\text{d},\Tilde{x}_0}\in\mathbb{R}^2$ is two-dimensional. Hence, its convex
    hull $\text{Conv}(P_{\text{d},\Tilde{x}_0})$ can be constructed by all convex combinations of two points of $P_{\text{d},\Tilde{x}_0}$. Hence, there must exist two (not necessarily distinct)} points $(C_1,V_1), (C_2,V_2) \in P_{\text{d},\Tilde{x}_0}$ associated to some $\pi_1,\pi_2\in \Pi_{\text{d}}$, such that $C = \eta C_1 + (1-\eta) C_2$ and $V = \eta V_1 + (1-\eta) V_2, \eta\in [0,1]$. Defining a policy $\pi_{\text{mix}}\in\Pi_{\text{mix}}$, consisting of $\pi_1,\pi_2\in \Pi_{\text{d}}$, which are sampled with probabilities $\eta$ and $1-\eta$, respectively, we obtain a mixed policy with expected cost $C$ and safety $V$. 
\end{proof}
\smallskip
Proposition \ref{thm_convexity_of_mixed_performance_sets} immediately yields the following.
\fi
\begin{corollary}
    \label{cor_mixing_two_policies_is_enough}
    Any point in the Performance Set $P_{\Pi_{\text{mix}},\Tilde{x}_0}$ can be attained by mixing at maximum two deterministic causal policies.
\end{corollary}
\rev{Corollary \ref{cor_mixing_two_policies_is_enough} has also been observed in \cite{Ono_4} and \cite{feinberg_1} and is a consequence of Proposition \ref{thm_convexity_of_mixed_performance_sets} and $P_{\Pi_{\text{mix}},\Tilde{x}_0}\subseteq\mathbb{R}^2$\cite{schmid2023computing}; the convex hull of a two-dimensional set can be constructed by the union of all convex combinations of two points in the set.}
\begin{lemma}[Attainability of Mixed Policies]
\label{thm_AttainabilityofMixedPolicies}
    Fix $\lambda\in\mathbb{R}_{\geq0}$. Then, there exists at least one measurable $\lambda$-optimal mixed policy $\pi_{\lambda}\in{\argmin}_{\pi\in\Pi_{\text{mix}}}  C_0^{\pi}(\Tilde{x}_0) - \lambda V_0^{\pi}(\Tilde{x}_0)$ that equals a $\lambda$-optimal deterministic Markov policy. Furthermore, any mixture of $\lambda$-optimal deterministic Markov policies is a $\lambda$-optimal mixed policy. 
\end{lemma}
\iffullversion
\begin{proof}
        By Corollary \ref{cor_mixing_two_policies_is_enough}, any point in $P_{\Pi_{\text{mix}},\Tilde{x}_0}$ can be constructed using two deterministic causal policies $\pi_1,\pi_2\in\Pi_{\text{d}}$ that are played with some probability $\eta\in[0,1]$ and $\eta-1$, respectively. Then,
\begin{align*}
    {\inf}_{\pi\in\Pi_{\text{mix}}}&  C_0^{\pi}(\Tilde{x}_0) - \lambda V_0^{\pi}(\Tilde{x}_0) \\= \ &{\inf}_{\pi_1,\pi_2\in\Pi_{\text{d}}}  \eta(C_0^{\pi_1}(\Tilde{x}_0) - \lambda V_0^{\pi_1}(\Tilde{x}_0)) \\&  + (1-\eta)(C_0^{\pi_2}(\Tilde{x}_0) - \lambda V_0^{\pi_2}(\Tilde{x}_0))
    \\= \ & {\inf}_{\pi_1\in\Pi_{\text{d}}}  \eta(C_0^{\pi_1}(\Tilde{x}_0) - \lambda V_0^{\pi_1}(\Tilde{x}_0)) \\&  + {\inf}_{\pi_2\in\Pi_{\text{d}}}(1-\eta)(C_0^{\pi_2}(\Tilde{x}_0) - \lambda V_0^{\pi_2}(\Tilde{x}_0))
    \\= \ & {\inf}_{\pi_1\in\Pi_{\text{dm}}}  \eta(C_0^{\pi_1}(\Tilde{x}_0) - \lambda V_0^{\pi_1}(\Tilde{x}_0)) \\&  + {\inf}_{\pi_2\in\Pi_{\text{dm}}}(1-\eta)(C_0^{\pi_2}(\Tilde{x}_0) - \lambda V_0^{\pi_2}(\Tilde{x}_0)),
    \\= \ & {\min}_{\pi_1\in\Pi_{\text{dm}}}  \eta(C_0^{\pi_1}(\Tilde{x}_0) - \lambda V_0^{\pi_1}(\Tilde{x}_0)) \\&  + {\min}_{\pi_2\in\Pi_{\text{dm}}}(1-\eta)(C_0^{\pi_2}(\Tilde{x}_0) - \lambda V_0^{\pi_2}(\Tilde{x}_0)),
\end{align*}
where the last two equalities follow from the problem being an MDP, allowing us to restrict attention to deterministic Markov policies, and Theorem \ref{thm_our_Dp-recursion}. If the minimizing deterministic Markov policy is unique, i.e., $\Pi_{\text{dm},\lambda}$ is a singleton, then the mixed policy is unique and equals the deterministic Markov policy. If it is non-unique, also any mixture of $\lambda$-optimal deterministic Markov policies yields a $\lambda$-optimal mixed policy. Since the minimizing $\lambda$-optimal deterministic Markov policies are measurable (see Theorem \ref{th_recursion}), the mixed policy is measurable. 
\end{proof}
\smallskip
\fi 
\begin{corollary}
      Any $\lambda$-optimal mixed policy has zero probability of sampling a deterministic Markov policy that is not $\lambda$-optimal. 
      \label{cor_mixedPolicyOnlySamplesLambdaOptimalDeterministic}
\end{corollary}
\iffullversion
    \begin{proof}
        Let $\pi_1\in\Pi_{\text{dm},\lambda}$ and $\pi_2\in\Pi_{\text{dm}}\setminus\Pi_{\text{dm},\lambda}$, i.e., $C_0^{\pi_1}(\Tilde{x}_0) - \lambda^{\star} V_0^{\pi_1}(\Tilde{x}_0)<C_0^{\pi_2}(\Tilde{x}_0) - \lambda^{\star} V_0^{\pi_2}(\Tilde{x}_0)$. Then,
        \begin{align*}
            {\min}_{\pi\in\Pi_{\text{mix}}}  &C_0^{\pi}(\Tilde{x}_0) - \lambda^{\star} V_0^{\pi}(\Tilde{x}_0) 
            \\= \ &C_0^{\pi_1}(\Tilde{x}_0) - \lambda^{\star} V_0^{\pi_1}(\Tilde{x}_0)
            \\< \ &\eta(C_0^{\pi_1}(\Tilde{x}_0) - \lambda^{\star} V_0^{\pi_1}(\Tilde{x}_0)) 
            \\ &+(1-\eta)(C_0^{\pi_2}(\Tilde{x}_0) - \lambda^{\star} V_0^{\pi_2}(\Tilde{x}_0))
        \end{align*}
        whenever $\eta < 1$, i.e., if there is a non-zero probability of sampling $\pi_2$. Hence, whenever there is a non-zero probability of sampling a deterministic Markov policy that is not $\lambda$-optimal, the resulting mixed policy can also not be $\lambda$-optimal since sampling any $\lambda$-optimal deterministic Markov policy with probability one yields a superior mixed policy.
    \end{proof}
\fi




\subsubsection{Equivalence of Problems}

\begin{lemma} (Monotonicity)
    The safety $V_0^{\pi_{\lambda}}(\Tilde{x}_0)$ and cost $C_0^{\pi_{\lambda}}(\Tilde{x}_0)$ of the policy $\pi_{\lambda}\in\Pi_{\text{mix},\lambda}$ are monotone in $\lambda\in\mathbb{R}_{\geq 0}$. Specifically, if $0\leq \lambda < \lambda',$ then $V_0^{\pi_\lambda}(\Tilde{x}_0)\leq V_0^{\pi_{\lambda'}}(\Tilde{x}_0)$ and $C_0^{\pi_\lambda}(\Tilde{x}_0)\leq C_0^{\pi_{\lambda'}}(\Tilde{x}_0)$ for all $\pi_{\lambda}\in\Pi_{\text{mix},\lambda}, \pi_{\lambda'}\in\Pi_{\text{mix},\lambda'}$. 
    \label{lem_monotonicity}
\end{lemma}
\iffullversion
\smallskip
\begin{proof}
\rev{Let $\lambda\geq 0$ be fixed throughout the proof and consider any $\pi_{\lambda}\in\Pi_{\text{mix},\lambda}$. 
By Definition 3.2, any $\lambda$-optimal mixed policy satisfies 
\begin{equation*}
    \Pi_{\text{mix},\lambda}=\{\pi\in\Pi_{\text{mix}}| \pi\in{\arg\min}_{\pi\in\Pi_{\text{mix}}}  C_0^{\pi}(\Tilde{x}_0) - \lambda V_0^{\pi}(\Tilde{x}_0) \},
\end{equation*} 
which leads to
\begin{align}
	C_0^{\pi_{\lambda}}(\Tilde{x}_0) - \lambda V_0^{\pi_{\lambda}}(\Tilde{x}_0)=\min_{\pi\in\Pi_{\text{mix}}}  C_0^{\pi}(\Tilde{x}_0) - \lambda V_0^{\pi}(\Tilde{x}_0).
 \label{eq_lambda_optimal_policy}
\end{align}
Hence, any policy with greater safety will be associated with greater control cost, and vice versa. This means that, for all $\pi\in \Pi_{\text{mix}}$,
\begin{align*}
    V_0^{\pi}(\Tilde{x}_0)> V_0^{\pi_{\lambda}}(\Tilde{x}_0) &\implies  C_0^{\pi}(\Tilde{x}_0)> C_0^{\pi_{\lambda}}(\Tilde{x}_0),\\
    C_0^{\pi}(\Tilde{x}_0)< C_0^{\pi_{\lambda}}(\Tilde{x}_0) &\implies  V_0^{\pi}(\Tilde{x}_0)< V_0^{\pi_{\lambda}}(\Tilde{x}_0).
\end{align*}
}
Otherwise, if there exists a policy $\hat{\pi}\in\Pi_{\text{mix}}$ that has greater or equal safety at a lower cost than $\pi_{\lambda}$, or lower or equal cost at a greater safety than $\pi_{\lambda}$, then  
\begin{equation}
    C_0^{\hat{\pi}}(\Tilde{x}_0) - \lambda V_0^{\hat{\pi}}(\Tilde{x}_0) < C_0^{\pi_{\lambda}}(\Tilde{x}_0) - \lambda V_0^{\pi_{\lambda}}(\Tilde{x}_0),
\end{equation}
which would be a contradiction to the definition of $\pi_{\lambda}$ being a minimizer in~\eqref{eq_lambda_optimal_policy}.

\rev{Now, consider any fixed $\lambda'\geq 0$, such that $0\leq\lambda < \lambda', \pi_{\lambda'}\in\Pi_{\text{mix},\lambda'}$. We aim to show that, if $\lambda < \lambda'$, then any $\pi_{\lambda'}\in\Pi_{\text{mix},\lambda'}$ will be at least as safe as $\pi_{\lambda}$ and have a cost no lower than $\pi_{\lambda}$, i.e., $V_0^{\pi_{\lambda'}}(\tilde x_0)\geq V_0^{\pi_{\lambda}}(\tilde x_0)$ and $C^{\pi_{\lambda'}}(\tilde x_0)\geq C^{\pi_{\lambda}}(\tilde x_0)$. We prove this by contradiction: The opposite of our claim is that either $V_0^{\pi_{\lambda'}}(\tilde x_0)< V_0^{\pi_{\lambda}}(\tilde x_0)$ or $C^{\pi_{\lambda'}}(\tilde x_0)< C^{\pi_{\lambda}}(\tilde x_0)$. To start, we assume that $V_0^{\pi_{\lambda'}}(\tilde x_0)< V_0^{\pi_{\lambda}}(\tilde x_0)$ and show that it leads to a contradiction. The steps can be followed analogously to show that $C^{\pi_{\lambda'}}(\tilde x_0)< C^{\pi_{\lambda}}(\tilde x_0)$ leads to a similar contradiction. Hence, it must hold that $V_0^{\pi_{\lambda'}}(\tilde x_0)\geq V_0^{\pi_{\lambda}}(\tilde x_0)$ and $C^{\pi_{\lambda'}}(\tilde x_0)\geq C^{\pi_{\lambda}}(\tilde x_0)$. 
}

Assume, for the sake of contradiction, that $V_0^{\pi_{\lambda}}(\Tilde{x}_0)> V_0^{\pi_{\lambda'}}(\Tilde{x}_0)$. Then, as above, $C_0^{\pi_{\lambda}}(\Tilde{x}_0) > C_0^{\pi_{\lambda'}}(\Tilde{x}_0)$.
\rev{
Otherwise, if $V_0^{\pi_{\lambda}}(\Tilde{x}_0)> V_0^{\pi_{\lambda'}}(\Tilde{x}_0)$ and $C_0^{\pi_{\lambda}}(\Tilde{x}_0) \leq C_0^{\pi_{\lambda'}}(\Tilde{x}_0)$, then
\begin{align*}
    C_0^{\pi_{\lambda}}(\Tilde{x}_0) - \lambda' V_0^{\pi_{\lambda}}(\Tilde{x}_0) < C_0^{\pi_{\lambda'}}(\Tilde{x}_0) - \lambda' V_0^{\pi_{\lambda'}}(\Tilde{x}_0),
\end{align*}
which would be contradicting the fact that
\begin{align*}
    \pi_{\lambda'}\in{\arg\min}_{\pi\in\Pi_{\text{mix}}}  C_0^{\pi}(\Tilde{x}_0) - \lambda' V_0^{\pi}(\Tilde{x}_0).
\end{align*}
}

Further, $V_0^{\pi_{\lambda}}(\Tilde{x}_0)>0$ since $V_0^{\pi_{\lambda}}(\Tilde{x}_0)> V_0^{\pi_{\lambda'}}(\Tilde{x}_0)\geq 0$. We then have that
\begin{align*}
    C_0^{\pi_{\lambda'}}(\Tilde{x}_0) &- \lambda' V_0^{\pi_{\lambda'}}(\Tilde{x}_0) \\&\leq C_0^{\pi_{\lambda}}(\Tilde{x}_0) - \lambda' V_0^{\pi_{\lambda}}(\Tilde{x}_0) \\
    &=C_0^{\pi_{\lambda}}(\Tilde{x}_0) - \lambda V_0^{\pi_{\lambda}}(\Tilde{x}_0) - (\lambda'-\lambda)V_0^{\pi_{\lambda}}(\Tilde{x}_0) \\
    &\leq C_0^{\pi_{\lambda'}}(\Tilde{x}_0) - \lambda V_0^{\pi_{\lambda'}}(\Tilde{x}_0) - (\lambda'-\lambda)V_0^{\pi_{\lambda}}(\Tilde{x}_0) \\
    &< C_0^{\pi_{\lambda'}}(\Tilde{x}_0) - \lambda V_0^{\pi_{\lambda'}}(\Tilde{x}_0) - (\lambda'-\lambda)V_0^{\pi_{\lambda'}}(\Tilde{x}_0) \\
    &= C_0^{\pi_{\lambda'}}(\Tilde{x}_0) - \lambda' V_0^{\pi_{\lambda'}}(\Tilde{x}_0),
\end{align*}
\rev{
The first inequality again holds by \begin{align*}
    \pi_{\lambda'}\in{\arg\min}_{\pi\in\Pi_{\text{mix}}}  C_0^{\pi}(\Tilde{x}_0) - \lambda' V_0^{\pi}(\Tilde{x}_0),
\end{align*}
the first equality holds since a zero-term was added, the second inequality holds by 
\begin{align*}
    \pi_{\lambda}\in{\arg\min}_{\pi\in\Pi_{\text{mix}}}  C_0^{\pi}(\Tilde{x}_0) - \lambda V_0^{\pi}(\Tilde{x}_0),
\end{align*}
the third inequality holds by $\lambda'>\lambda\geq 0$, i.e., $-(\lambda'-\lambda)<0$, and by the assumption that $V_0^{\pi_{\lambda}}(\Tilde{x}_0)> V_0^{\pi_{\lambda'}}(\Tilde{x}_0)$, which we posed for contradiction. However, the statement is a contradiction since it claims  
\begin{equation}
     C_0^{\pi_{\lambda'}}(\Tilde{x}_0) - \lambda' V_0^{\pi_{\lambda'}}(\Tilde{x}_0) < C_0^{\pi_{\lambda'}}(\Tilde{x}_0) - \lambda' V_0^{\pi_{\lambda'}}(\Tilde{x}_0).
\end{equation}
Hence, the initial hypothesis must be wrong and $V_0^{\pi_{\lambda}}(\Tilde{x}_0)> V_0^{\pi_{\lambda'}}(\Tilde{x}_0)$ can not hold. Therefore, it must hold that $V_0^{\pi_{\lambda}}(\Tilde{x}_0)\leq V_0^{\pi_{\lambda'}}(\Tilde{x}_0)$.

The same line of arguments would apply if we were to assume $C_0^{\pi_{\lambda}}(\Tilde{x}_0) > C_0^{\pi_{\lambda'}}(\Tilde{x}_0)$ as it leads to a similar contradiction. The proof steps can be followed analogously.

Hence, overall $V_0^{\pi_\lambda}(\Tilde{x}_0)\leq V_0^{\pi_{\lambda'}}(\Tilde{x}_0)$ and $C_0^{\pi_\lambda}(\Tilde{x}_0)\leq C_0^{\pi_{\lambda'}}(\Tilde{x}_0)$. Since this holds for any $\pi_{\lambda}\in\Pi_{\text{mix},\lambda},\pi_{\lambda'}\in\Pi_{\text{mix},\lambda'}$, monotonicity follows.
}
\end{proof}
\smallskip
\fi
\begin{remark}
\label{remark_mon_lemma}
    Lemma \ref{lem_monotonicity} also holds for $\Pi_{\text{dm},\lambda}$.  
\end{remark}





\ifsingleCol
    \begin{figure}
        \centering
        \includegraphics[height=3cm]{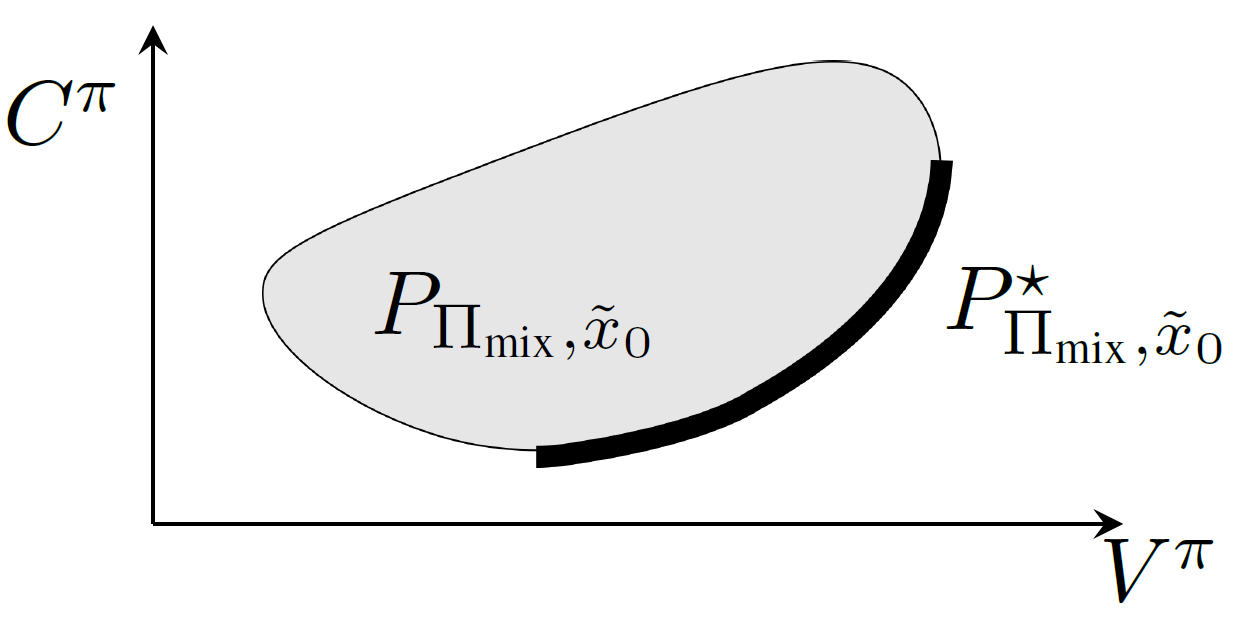}
        \caption{Illustration of a Performance Set $P_{\Pi_{\text{mix}},\Tilde{x}_0}$ and its Pareto front $P_{\Pi_{\text{mix}},\Tilde{x}_0}^{\star}$.}
        \label{fig_pareto_front_definition}
    \end{figure}
\else
    \ifNonArxivVersion
    \begin{figure}
        \centering
        \begin{tikzpicture}
        \begin{axis}[width=0.6\columnwidth,height=4cm,%
                    xlabel={$V^{\pi}$}, ylabel=$C^{\pi}$, ylabel shift = -.1cm, xlabel shift = -.1cm, ticks=none, xtick style={draw=none}, ytick style={draw=none},axis x line = bottom,axis y line = left, axis lines=middle, xlabel style={at={(0,0)}, right,yshift=-5pt,xshift=100pt}, ylabel style={at={(0,0)}, right,yshift=45pt,xshift=-20pt},ymin=0,xmin=0,ymax=1.5,xmax=1.5]
            \end{axis}
            \node [] (label) at (1.75,1) {\includegraphics[width=.3\columnwidth]{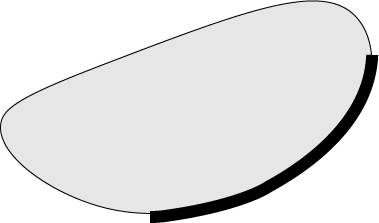} };
            \draw (3.6,0.8) node {$P_{\Pi_{\text{mix}},\Tilde{x}_0}^{\star}$};
            \draw (1.4,0.8) node {$P_{\Pi_{\text{mix}},\Tilde{x}_0}$};
        \end{tikzpicture}
        \caption{A Performance Set $P_{\Pi_{\text{mix}},\Tilde{x}_0}$ and Pareto front $P_{\Pi_{\text{mix}},\Tilde{x}_0}^{\star}$.}
        \label{fig_pareto_front_definition}
    \end{figure}
    \else
        \begin{figure}
            \centering
            \includegraphics[height=2.5cm]{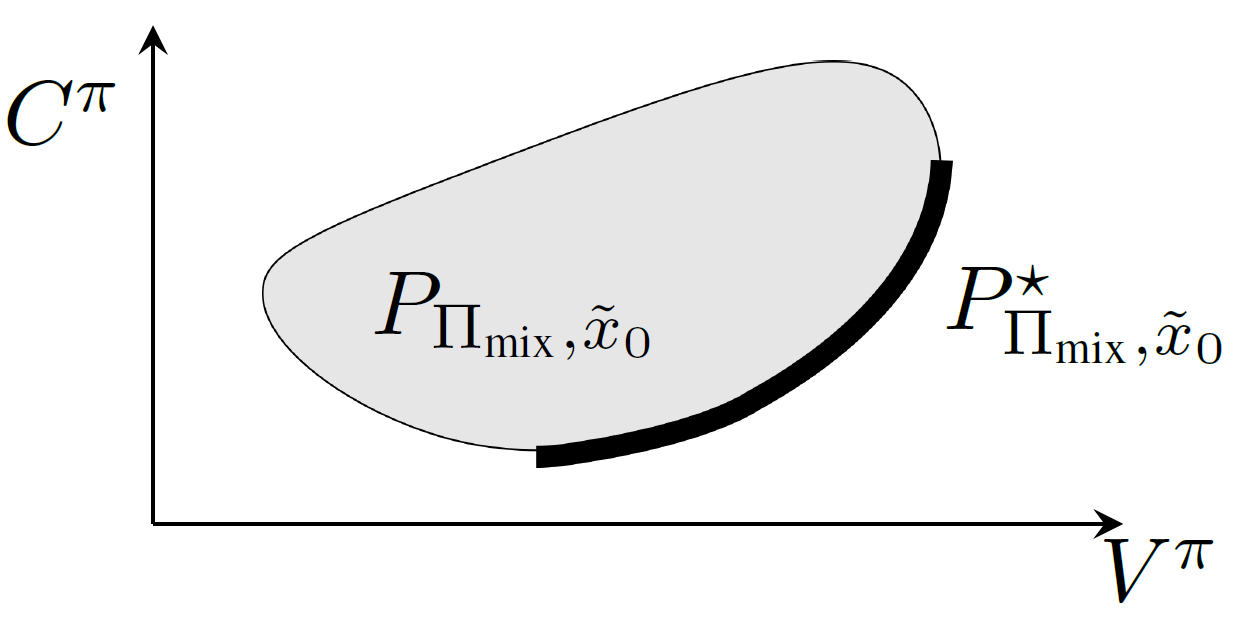}
            \caption{Illustration of a Performance Set $P_{\Pi_{\text{mix}},\Tilde{x}_0}$ and its Pareto front $P_{\Pi_{\text{mix}},\Tilde{x}_0}^{\star}$.}
            \label{fig_pareto_front_definition}
        \end{figure}
    \fi
\fi

\begin{theorem}[Equivalence of Problems \eqref{eq_problemFormulation_CCOC} and \eqref{bilevel}]
    Assume that the minimum in \eqref{bilevel} is attained by $\lambda^{\star}\in[0,\infty)$. Then, there exists $\pi_{\lambda^{\star}}\in\Pi_{\text{mix},\lambda^{\star}}$ such that $V_0^{\pi_{\lambda^{\star}}}(\Tilde{x}_0)\geq \alpha$ and $\pi_{\lambda^{\star}}$ is a minimizing argument of Problem \eqref{eq_problemFormulation_CCOC}.
    \label{thm_Optimality_of_new_ccoc_formulation}
\end{theorem}
\begin{proof} 
We assume $\lambda^{\star}$ is attained, i.e., the constraints in Problem \eqref{bilevel} must hold. Hence, there exists at least one $\lambda^{\star}$-optimal policy $\pi_{\lambda^{\star}}\in\Pi_{\text{mix},\lambda^{\star}}$ which yields $V_0^{\pi_{\lambda^{\star}}}(\Tilde{x}_0)\geq \alpha$. 

Exploiting the equivalence of mixed and stochastic causal policies, let $\pi^{\star}\in\Pi_{\text{mix}}$ be an optimal argument of Problem \eqref{eq_problemFormulation_CCOC} with associated control cost $C_0^{\pi^{\star}}(\Tilde{x}_0)$ and safety $V_0^{\pi^{\star}}(\Tilde{x}_0)$. Assume, for the sake of contradiction, that $\pi_{\lambda^{\star}}$ is not an optimal argument of Problem \eqref{eq_problemFormulation_CCOC}. Then, $C_0^{\pi^{\star}}(\Tilde{x}_0)<C_0^{\pi_{\lambda^{\star}}}(\Tilde{x}_0)$ by optimality of $\pi^{\star}$ in Problem \eqref{eq_problemFormulation_CCOC} and consequently $V_0^{\pi^{\star}}(\Tilde{x}_0) < V_0^{\pi_{\lambda^{\star}}}(\Tilde{x}_0)$ by optimality of $\pi_{\lambda^{\star}}$ in Problem \eqref{bilevel}.

If there is a $\lambda$ leading to $\pi^{\star}\in\Pi_{\text{mix},\lambda}$, we must have that for any $\overline{\pi},\underline{\pi}\in\Pi_{\text{mix}}$ with 
$V_0^{\underline{\pi}}(\Tilde{x}_0)\leq V_0^{\pi^{\star}}(\Tilde{x}_0)\leq V_0^{\overline{\pi}}(\Tilde{x}_0)$,
\begin{align*}
    C_0^{\pi^{\star}}(\Tilde{x}_0)-\lambda V_0^{\pi^{\star}}(\Tilde{x}_0) &\leq C_0^{\overline{\pi}}(\Tilde{x}_0)-\lambda V_0^{\overline{\pi}}(\Tilde{x}_0) \\ \Leftrightarrow \lambda V_0^{\overline{\pi}}(\Tilde{x}_0)-\lambda V_0^{\pi^{\star}}(\Tilde{x}_0) &\leq C_0^{\overline{\pi}}(\Tilde{x}_0)-C_0^{\pi^{\star}}(\Tilde{x}_0)  \\
    \intertext{and similarly}
     C_0^{\pi^{\star}}(\Tilde{x}_0)-\lambda V_0^{\pi^{\star}}(\Tilde{x}_0) &\leq C_0^{\underline{\pi}}(\Tilde{x}_0)-\lambda V_0^{\underline{\pi}}(\Tilde{x}_0) \\ \Leftrightarrow \lambda V_0^{\underline{\pi}}(\Tilde{x}_0)-\lambda V_0^{\pi^{\star}}(\Tilde{x}_0) &\leq  C_0^{\underline{\pi}}(\Tilde{x}_0)-C_0^{\pi^{\star}}(\Tilde{x}_0),\\
    \intertext{which combined leads to}
    \frac{C_0^{\underline{\pi}}(\Tilde{x}_0)-C_0^{\pi^{\star}}(\Tilde{x}_0)}{V_0^{\underline{\pi}}(\Tilde{x}_0)-V_0^{\pi^{\star}}(\Tilde{x}_0)} &\leq  \ \lambda \leq \frac{C_0^{\overline{\pi}}(\Tilde{x}_0)-C_0^{\pi^{\star}}(\Tilde{x}_0)}{V_0^{\overline{\pi}}(\Tilde{x}_0)-V_0^{\pi^{\star}}(\Tilde{x}_0)}.
\end{align*}
Under these bounds a feasible $\lambda$ always exists if the Pareto front is convex, which is indeed the case for mixed policies by Proposition \ref{thm_convexity_of_mixed_performance_sets}. 

However, then $\lambda<\lambda^{\star}$ by Lemma \ref{lem_monotonicity}. 
This violates the assumption that $\lambda^{\star}$ is the attained minimum. Hence $\pi_{\lambda^{\star}}$ is an optimal argument for Problem \eqref{eq_problemFormulation_CCOC}.
\end{proof} 
\rev{Recall that \eqref{eq_lambda_optimal_policy_eq} is the Lagrange dual of \eqref{eq_problemFormulation_CCOC} and equivalent to \eqref{bilevel}. Hence, the equivalence of \eqref{eq_problemFormulation_CCOC} and \eqref{bilevel} implies strong duality of \eqref{eq_problemFormulation_CCOC} and \eqref{eq_lambda_optimal_policy_eq}. Intuitively, strong duality holds because the Pareto front $P_{\text{mix},\Tilde{x}_0}^{\star}(\cdot)$ is convex under mixed policies and we can achieve any extremum point of this set as a deterministic Markov policy (or mixture thereof). Therefore, by pushing the safety constraint into the cost function of the Lagrange dual we can use DP to solve the resulting, potentially non-convex, minimization problem. This is in contrast to convex methods that directly consider the constraints in the optimization problem like \cite{schmid_2,paulson,patil_1,wang_2,raghuraman,bavdekar}. While these methods are computationally superior, they conservatively approximate the joint chance constraints and require convex safe sets, as well as specific classes of dynamics and noise distributions. For instance, consider the stochastic Model Predictive Controller in \cite{mesbah2016stochastic}. Given non-convex constraints, strong duality with its Lagrange dual may not hold. In fact, \cite{mesbah2016stochastic} computes input sequences, which can be interpreted as a trivial policy applying the same input
to all states at each point in time. This would be equivalent to optimizing over a subclass of
deterministic Markov policies; our results suggest that the resulting Performance Set might not be convex, leading to
a duality gap, and that optimizing over mixtures of input trajectories could recover strong duality.}

Interestingly, if the Pareto front $P_{\text{mix},\Tilde{x}_0}^{\star}(\cdot)$ is strictly convex, then the set of $\lambda^{\star}$-optimal deterministic Markov policies will become a singleton. Hence, the optimal mixed policy $\pi_{\lambda^{\star}}$ will be equivalent to a deterministic Markov policy. \rev{Conversely, when restricting Problem \eqref{eq_problemFormulation_CCOC} to deterministic Markov policies strong duality does not necessarily hold since the Performance Set of deterministic Markov policies might be non-convex.}

\section{Algorithmic Solution}\label{sec_algorithm}

\iffullversion
Next, we design an optimization procedure to solve Problem \eqref{bilevel} and provide conditions to check feasibility. 
\subsection{Feasibility Check and Boundary Solutions}
\fi
\rev{Following \cite{Ono_2}, before solving Problem \eqref{eq_problemFormulation_CCOC}, one might first check its feasibility and triviality. By the latter we mean that the minimum cost policy is already safe enough. This can be checked by running recursion \eqref{eq_max_invariance_evaluation}, restricting to inputs optimal under the value functions obtained from \eqref{eq_max_cost_evaluation}. 
On the contrary, if the policy associated with the maximum safety recursion \eqref{eq_max_invariance_evaluation} yields a safety less than $\alpha$, then Problem \eqref{eq_problemFormulation_CCOC} is infeasible.} %
\iffullversion

    Note that the policies associated with \eqref{eq_max_cost_evaluation} and \eqref{eq_max_invariance_evaluation} are not necessarily unique and the recursions either only optimize for control cost or safety, respectively. 
    However, it is still desirable to recover a preference for policies that, e.g., yield higher safety at the same control cost or lower control cost at the same safety. Therefore, the following Proposition is useful (see also Fig.~\ref{fig_border_case_examples}). 
    
    
    
    
    \ifsingleCol
        \begin{figure}[t!]
        \centering
            \includegraphics[height=3cm]{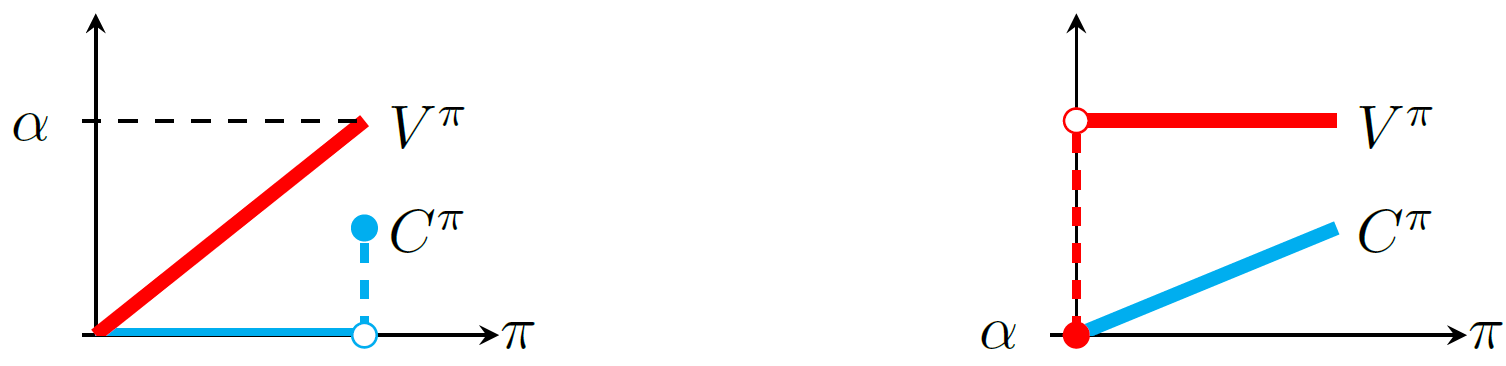}
            \caption{In the left plot, we can choose from policies which have a safety arbitrarily close to $\alpha$ and a control cost of $C$ or a policy that attains $\alpha$ but at cost $C+\delta$, $\delta>0$. Then, for any $\lambda$, there always exists a policy $\pi$ with safety close enough to $\alpha$ such that it is not optimal to incur the additional cost $\delta$, i.e., $\lambda(\alpha-V^{\pi})<\delta$. The right plot depicts a similar border case.}
            \label{fig_border_case_examples}
        \end{figure}
    \else
        \begin{figure}[t!]
        \centering
            \begin{tikzpicture}
                \begin{axis}[width=0.45\columnwidth,%
                            xlabel={$\pi$}, ylabel shift = -.1cm, xlabel shift = -.1cm, xtick={0,2}, ytick={0,2}, xtick style={draw=none}, ytick style={draw=none},axis x line = bottom,axis y line = left, 
            axis lines=middle, xlabel style={at={(0,0)}, right,yshift=0pt,xshift=65pt}, ylabel style={at={(0,0)}, right,yshift=55pt,xshift=-15pt},ymax=1.5,xmax=1.5]
                \addplot +[line width=0.8mm, mark=none,cyan] coordinates {(0,0) (1,0)};
                \addplot +[line width=0.8mm, mark=none,red] coordinates {(0,0) (1,1)};
                \addplot +[mark=none,black, dashed] coordinates {(-.05,1) (1,1)};
                \addplot[bholdot] coordinates{(1,0)};
                \addplot[bsoldot] coordinates{(1,.5)};
                \addplot +[mark=none,cyan, dashed, very thick] coordinates {(1,0) (1,.5)};
                \end{axis}
                \draw (2.3,1.2) node {$V^{\pi}_0(\Tilde{x}_0)$};
                \draw (2.3,0.6) node {$C^{\pi}_0(\Tilde{x}_0)$};
                \draw (-.3,1.2) node {$\alpha$};
            \end{tikzpicture}
            \hfill
            \begin{tikzpicture}
                \begin{axis}[width=0.45\columnwidth,%
                            xlabel={$\pi$}, ylabel shift = -.1cm, xlabel shift = -.1cm, xtick={0,2}, ytick={0,2}, xtick style={draw=none}, ytick style={draw=none},axis x line = bottom,axis y line = left, 
            axis lines=middle, xlabel style={at={(0,0)}, right,yshift=0pt,xshift=65pt}, ylabel style={at={(0,0)}, right,yshift=55pt,xshift=-15pt},ymax=1.5,xmax=1.5]
                \addplot +[line width=0.8mm, mark=none,red] coordinates {(0,1) (1,1)};
                \addplot +[line width=0.8mm, mark=none,cyan] coordinates {(0,0) (1,.5)};
                \addplot +[mark=none,black, dashed] coordinates {(-.1,0) (0,0)};
                \addplot[rholdot] coordinates{(0,1)};
                \addplot[rsoldot] coordinates{(0,0)};
                \addplot +[mark=none,red, dashed, very thick] coordinates {(0,0) (0,1)};
                \end{axis}
                \draw (2.3,1.2) node {$V^{\pi}_0(\Tilde{x}_0)$};
                \draw (2.3,0.6) node {$C^{\pi}_0(\Tilde{x}_0)$};
                \draw (-.3,0) node {$\alpha$};
            \end{tikzpicture}
            \caption{In the left plot, we can choose from policies which have a safety arbitrarily close to $\alpha$ and a control cost of $C$ or a policy that attains $\alpha$ but at cost $C+\delta$, $\delta>0$. Then, for any $\lambda$, there always exists a policy $\pi$ with safety close enough to $\alpha$ such that it is not optimal to incur the additional cost $\delta$, i.e., $\lambda(\alpha-V^{\pi}_0)<\delta$. The right plot depicts a similar border case.}
            \label{fig_border_case_examples}
        \end{figure}
    \fi

    \smallskip
    \begin{proposition} (Border Case Suboptimality)
    \label{pro_Border_cases_method} 
    Take any $\lambda>0$. Then, 
    \begin{align*}
        C^{\pi_{\lambda}}_0(\Tilde{x}_0)-\underline{C}(\Tilde{x}_0)&\leq \lambda (V^{\pi_{\lambda}}_0(\Tilde{x}_0)- \underline{V}(\Tilde{x}_0)),
    \end{align*}
    which goes to zero as $\lambda$ goes to zero.
    Moreover, if $\overline{C}(\Tilde{x}_0)$ is bounded. Then, 
        \begin{align*}
        \overline{V}(\Tilde{x}_0)- V^{\pi_{\lambda}}_0(\Tilde{x}_0)&\leq \frac{\overline{C}(\Tilde{x}_0)-C^{\pi_{\lambda}}_0(\Tilde{x}_0)}{\lambda},
    \end{align*} 
    which goes to zero as $\lambda$ goes to infinity.
    \end{proposition}
    \smallskip
    \begin{proof} 
    To show the first statement, fix $\lambda>0$. By $\pi_{\lambda}$ being $\lambda$-optimal, we have
        \begin{align*}
            C^{\pi_{\lambda}}_0(\Tilde{x}_0)-\lambda V^{\pi_{\lambda}}_0(\Tilde{x}_0)&\leq \underline{C}(\Tilde{x}_0)-\lambda \underline{V}(\Tilde{x}_0).
        \end{align*}
       Rearranging above equation yields the bound. Taking $\lambda$ to zero, convergence to zero is guaranteed since $V^{\pi_{\lambda}}_0(\Tilde{x}_0), \underline{V}(\Tilde{x}_0)$ are bounded between zero and one. For the second statement, by $\pi_{\lambda}$ being $\lambda$-optimal we have
        \begin{align*}
            C^{\pi_{\lambda}}_0(\Tilde{x}_0)-\lambda V^{\pi_{\lambda}}_0(\Tilde{x}_0)&\leq \overline{C}(\Tilde{x}_0)-\lambda \overline{V}(\Tilde{x}_0)
    \end{align*}
    Combined with $\overline{V}(\Tilde{x}_0)\geq V^{\pi_{\lambda}}_0(\Tilde{x}_0)$, we have $\overline{C}(\Tilde{x}_0)\geq C^{\pi_{\lambda}}_0(\Tilde{x}_0)$.
    Rearranging above equation yields the bound, which goes to zero with $\lambda$ going to infinity since we assumed $\overline{C}(\Tilde{x}_0)$ to be bounded and $\overline{C}(\Tilde{x}_0)\geq C^{\pi_{\lambda}}_0(\Tilde{x}_0) \geq 0$.
    \end{proof}
    
    Most oftenly, however, the maximum safe policy is too costly and the minimum cost policy not safe enough, leading to a trade-off. We do not explore this trade-off much, but aim directly for a solution that attains the minimum cost at the required safety.
    
    \subsection{Bisection Algorithm}
\fi
Assume that Problem \eqref{eq_problemFormulation_CCOC} is neither infeasible nor trivial, let $\lambda^{\star}$ be the optimal solution to Problem \eqref{bilevel} and recall that any associated $\lambda^{\star}$-optimal mixed policy is constructed from a set of $\lambda^{\star}$-optimal deterministic Markov policies. 
Consider $\pi_1,\pi_2\in \Pi_{\text{dm},\lambda^{\star}}$, sampled with probabilities $\mu, 1-\mu$, respectively. Then, the safety of the corresponding mixed policy $\pi_{\text{mix}}$ is 
 $   V_{0}^{\pi_{\text{mix}}}(x_0) = \mu V_{0}^{\pi_1}(x_0) + (1-\mu) V_{0}^{\pi_2}(x_0)$.
For $\lambda^{\star}$ to be the optimal solution to Problem \eqref{bilevel}, we must have $V_{0}^{\pi_{\text{mix}}}(x_0)\geq \alpha$, and hence there must exist at least one policy $\pi\in\Pi_{\text{dm},\lambda^{\star}}$, which has a safety of at least $\alpha$.

Now, by optimality of Problem \eqref{bilevel}, for any $\underline{\lambda}\in\mathbb{R}_{\geq 0}$ with $\underline{\lambda}<\lambda^{\star}$, we have that $V_{0}^{\pi_{\text{mix}}}(x_0)<\alpha$. This implies that there does not exist $\pi\in\Pi_{\text{dm},\underline{\lambda}}$ with $V_{0}^{\pi}(x_0)\geq \alpha$. Otherwise, choosing $\pi_{\text{mix}}$ to sample $\pi$ with probability one would yield a feasible mixed policy. On the other hand, for any $\overline{\lambda}$ with $\overline{\lambda}>\lambda^{\star}$, every $\overline{\lambda}$-optimal deterministic Markov policy has a safety of at least $\alpha$, since there exists a policy in $\Pi_{\text{dm},\lambda^{\star}}$ with a safety of at least $\alpha$ (Remark \ref{remark_mon_lemma}). 

Hence, based on two policies $\pi_{\underline{\lambda}}\in\Pi_{\text{dm},\underline{\lambda}}$ and $\pi_{\overline{\lambda}}\in\Pi_{\text{dm},\overline{\lambda}}$, we can construct a mixed policy with safety greater or equal to $\alpha$, as assigning probability one to $\pi_{\overline{\lambda}}$ and zero to $\pi_{\underline{\lambda}}$ is always feasible. Ideally, one would assign an as high as possible probability measure to $\pi_{\underline{\lambda}}$, since it is associated with a lower control cost (Remark \ref{remark_mon_lemma}). This is achieved by linearly interpolating between $\pi_{\underline{\lambda}}$ and $\pi_{\overline{\lambda}}$ in such a way, that, in expectation, we obtain a safety of exactly $\alpha=p_{\overline{\lambda}}V_0^{\pi_{\overline{\lambda}}}(\Tilde{x}_0) + (1-p_{\overline{\lambda}})V_0^{\pi_{\underline{\lambda}}}(\Tilde{x}_0)$ (see Fig.~\ref{fig_outer_loop_opt_and_stochastic_policies}).
This in turn can be achieved by the policy
\begin{align}
\label{eq_stochastic_policies_ratios}
    \pi_{\text{mix}}=\begin{cases}
\pi_{\overline{\lambda}} & \text{with probability } p_{\overline{\lambda}}\!=\!\frac{\alpha - V_0^{\pi_{\underline{\lambda}}}(\Tilde{x}_0)}{V_0^{\pi_{\overline{\lambda}}}(\Tilde{x}_0) - V_0^{\pi_{\underline{\lambda}}}(\Tilde{x}_0)}, \\
\pi_{\underline{\lambda}} & \, \text{otherwise.}
\end{cases}
\end{align}

\ifsingleCol
    \begin{figure}
        \centering
        \includegraphics[height=6cm]{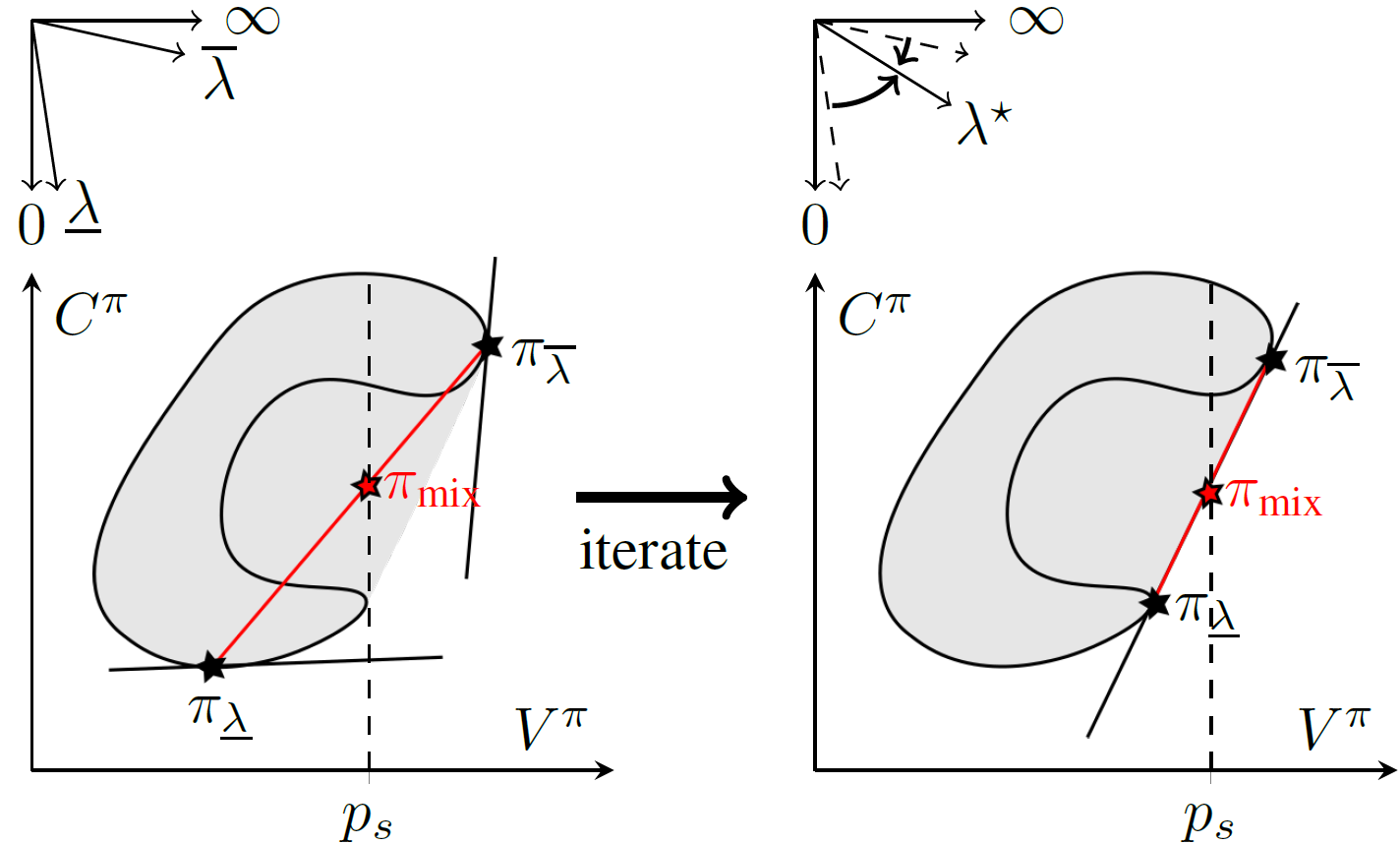}
        \caption{Performance sets $P_{\Pi_{\text{d}},\Tilde{x}_0}$ (bordered set) and its convex hull $P_{\Pi_{\text{mix}},\Tilde{x}_0}$ (grey set), as well as the performance of the respective policies $\underline{\lambda},\overline{\lambda}$ (black stars) and the optimal interpolation $\pi_{\text{mix}}$ according to equation \eqref{eq_stochastic_policies_ratios} (red star). The variable $\lambda$ defines the optimization direction. The DP recursion returns the optimal policy in the performance set in this direction. Along the outer loop iterations $\underline{\lambda}$ approaches $\overline{\lambda}$.}
        \label{fig_outer_loop_opt_and_stochastic_policies}
    \end{figure}
\else
    \ifNonArxivVersion
        \begin{figure}
    \centering
    \resizebox{0.8\columnwidth}{!}{%
    \hspace{-2.2cm}
    \begin{tikzpicture}
        \node [] (label) at (3.5,1.8) {\includegraphics[width=0.8\columnwidth]{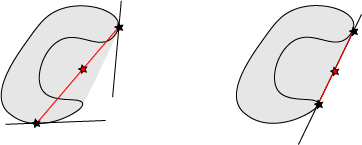} };
        
        \draw (0.7,0.5) node {$\pi_{\underline{\lambda}}$};
        \draw (2.6,2.6) node {$\pi_{\overline{\lambda}}$};
        \draw (1.96,1.85) node {\color{red}$\pi_{\text{mix}}$\color{black}};
        
        \draw (6.5,1.1) node {$\pi_{\underline{\lambda}}$};
        \draw (6.9,1.8) node {\color{red}$\pi_{\text{mix}}$\color{black}};
        \draw (7.2,2.5) node {$\pi_{\overline{\lambda}}$};

        \node [] (label) at (-1.2,0.2) {
        \begin{axis}[width=5cm,%
                xlabel={$V^{\pi}$}, ylabel=$C^{\pi}$,axis lines=middle, height=4.5 cm,xmin=0,xmax=1,xtick={0,.58}, ytick={0}, xticklabels={0,$p_s$}]
            \addplot +[mark=none,black, dashed] coordinates {(.58,0) (.58,1)};        
        \end{axis}};
            \node [] (label) at (3.4,0.2) {
        \begin{axis}[width=5cm,%
                xlabel={$V^{\pi}$}, ylabel=$C^{\pi}$,axis lines=middle, height=4.5 cm,xmin=0,xmax=1,xtick={0,.68}, ytick={0}, xticklabels={0,$p_s$}]
            \addplot +[mark=none,black, dashed] coordinates {(.68,0) (.68,1)};    
        \end{axis}};

        
        \coordinate (compass) at (-1.2,1.3);
        \draw[->] ($(compass) + (0.8,3.3)$) -- ($(compass) + (1.8,3.3)$) ;
        \draw[->] ($(compass) + (0.8,3.3)$) -- ($(compass) + (0.8,2.3)$);
        \draw ($(compass) + (0.8,2.1)$) node {$0$};
        \draw ($(compass) + (2.1,3.3)$) node {$\infty$};
        \draw[->] ($(compass) + (0.8,3.3)$) -- ($(compass) + (0.95,2.3)$);
        \draw ($(compass) + (1.1,2.2)$) node {$\underline{\lambda}$};
        \draw[->] ($(compass) + (0.8,3.3)$) -- ($(compass) + (1.7,3.1)$);
        \draw ($(compass) + (1.9,3)$) node {$\overline{\lambda}$};

        \draw[ultra thick, ->] (2.8,1.8) -- (3.8,1.8) ;
        \draw (3.25,1.5) node {iterate};

        \coordinate (compass) at (3.4,1.3);
        \draw[->] ($(compass) + (0.8,3.3)$) -- ($(compass) + (1.8,3.3)$) ;
        \draw[->] ($(compass) + (0.8,3.3)$) -- ($(compass) + (0.8,2.3)$);
        \draw[->, dashed] ($(compass) + (0.8,3.3)$) -- ($(compass) + (0.95,2.3)$);
        \draw[->, dashed] ($(compass) + (0.8,3.3)$) -- ($(compass) + (1.7,3.1)$);
        \draw[thick, ->] ($(compass) + (0.9,2.8)$) arc (-90:-40:.5);
        \draw[thick, ->] ($(compass) + (1.35,3.2)$) arc (-10:-30:.5);
        \draw ($(compass) + (0.8,2.1)$) node {$0$};
        \draw ($(compass) + (2.1,3.3)$) node {$\infty$};
        \draw[->] ($(compass) + (0.8,3.3)$) -- ($(compass) + (1.6,2.8)$);
        \draw ($(compass) + (1.8,2.7)$) node {$\lambda^{\star}$};
    \end{tikzpicture}
    }
    \caption{Performance sets $P_{\Pi_{\text{d}},\Tilde{x}_0}$ (bordered set) and its convex hull $P_{\Pi_{\text{mix}},\Tilde{x}_0}$ (grey set), as well as the performance of the respective policies $\pi_{\underline{\lambda}},\pi_{\overline{\lambda}}$ (black stars), \rev{the performance of all mixed policies constructable from $\pi_{\underline{\lambda}}$ and $\pi_{\overline{\lambda}}$ (red line),} and the optimal interpolation $\pi_{\text{mix}}$ according to equation \eqref{eq_stochastic_policies_ratios} (red star). The variable $\lambda$ sets the optimization direction. The DP recursion returns the optimal policy in the performance set in this direction. Over the outer loop iterations $\underline{\lambda}$ approaches $\overline{\lambda}$.}
    \label{fig_outer_loop_opt_and_stochastic_policies}
\end{figure}
    \else
        \begin{figure}
            \centering
            \includegraphics[height=5cm]{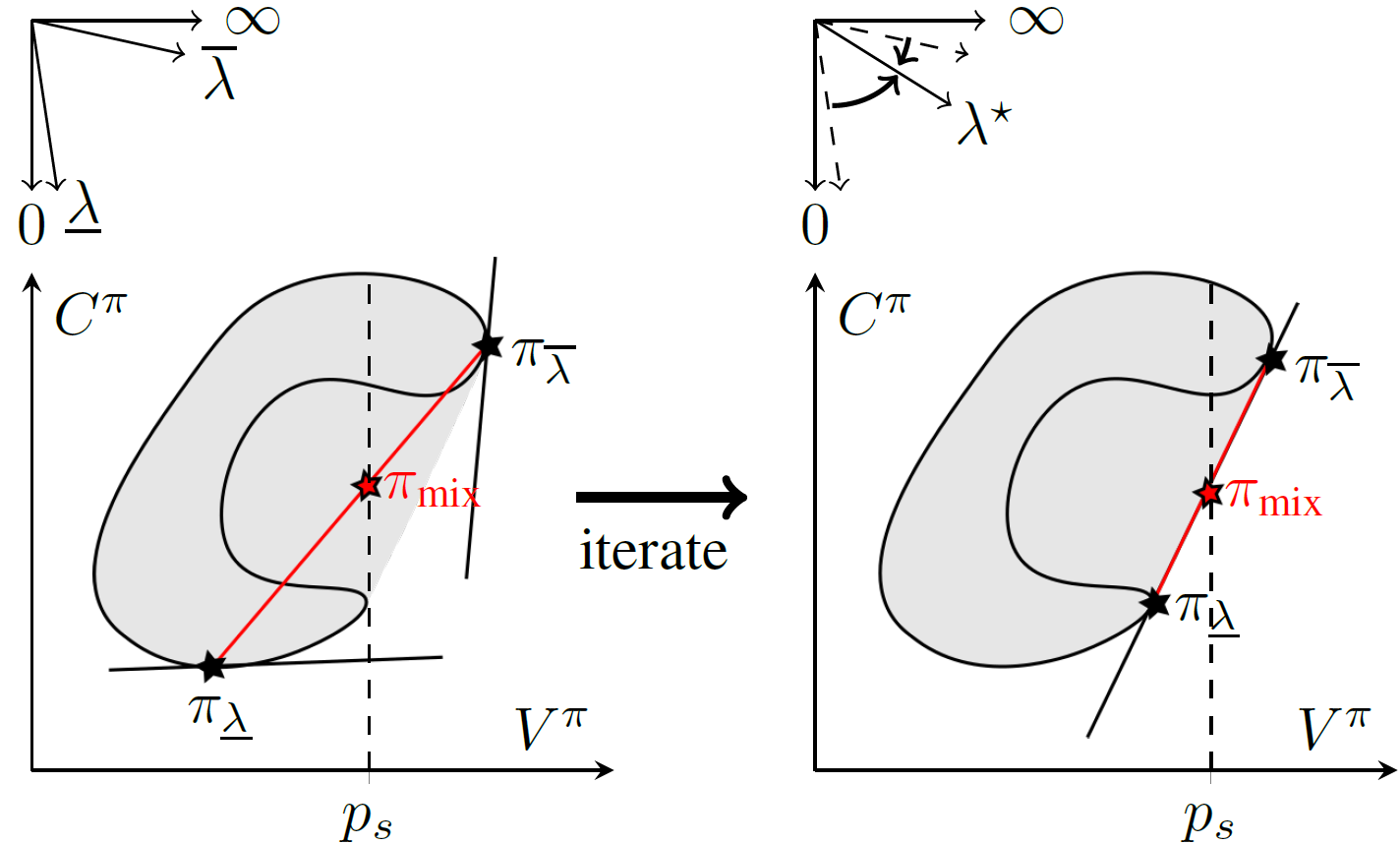}
            \caption{Performance sets $P_{\Pi_{\text{d}},\Tilde{x}_0}$ (bordered set) and its convex hull $P_{\Pi_{\text{mix}},\Tilde{x}_0}$ (grey set), as well as the performance of the respective policies $\underline{\lambda},\overline{\lambda}$ (black stars) and the optimal interpolation $\pi_{\text{mix}}$ according to equation \eqref{eq_stochastic_policies_ratios} (red star). The variable $\lambda$ defines the optimization direction. The DP recursion returns the optimal policy in the performance set in this direction. Along the outer loop iterations $\underline{\lambda}$ approaches $\overline{\lambda}$.}
            \label{fig_outer_loop_opt_and_stochastic_policies}
        \end{figure}
    \fi
\fi


Following \cite{Ono_2}, we propose a bisection algorithm to shrink bounds $[\underline{\lambda},\overline{\lambda}]$ on $\lambda^{\star}$. 
Recursions \eqref{eq_max_cost_evaluation} and \eqref{eq_max_invariance_evaluation} allow us to provide an initial range $[\underline{\lambda}_{\text{init}},\overline{\lambda}_{\text{init}}]$ of values for $\lambda^{\star}$.
\smallskip
\begin{proposition}[Bounds on $\lambda^{\star}$, \cite{pfeiffer}]
The value $\underline{\lambda}_{\text{init}}=0$ always yields a lower bound on $\lambda^{\star}\in\mathbb{R}_{\geq0}$. Assume $\overline{C}(\Tilde{x}_0)$ is bounded and $\overline{V}(\Tilde{x}_0)>\alpha$ and choose $\overline{\lambda}_{\text{init}}=\frac{\overline{C}(\Tilde{x}_0) - \underline{C}(\Tilde{x}_0) }{\overline{V}(\Tilde{x}_0)-\alpha}$. Then, the policy $\pi_{\overline{\lambda}_{\text{init}}}$ achieves a safety greater than $\alpha$. 
\label{thm_upper_lower_bound_on_lambda}
\end{proposition}
\smallskip
Note that \cite{pfeiffer} considers deterministic Markov policies. However, since any $\overline{\lambda}_{\text{init}}$-optimal mixed policy can be constructed from $\overline{\lambda}_{\text{init}}$-optimal deterministic Markov policies, $\overline{\lambda}_{\text{init}}$ is feasible, yet possibly suboptimal for Problem \eqref{bilevel}.

Starting with the range provided in Proposition \ref{thm_upper_lower_bound_on_lambda} we can now solve for $\lambda^{\star}$ using bisection, shrinking the bounds $\underline{\lambda}\leq\lambda^{\star}\leq\overline{\lambda}$ \cite{Ono_2}. 
The test whether $\lambda$ is greater or lower than $\lambda^{\star}$ is performed by evaluating $V_0^{\pi_{\lambda}}$ and comparing it with $\alpha$ (see Lemma \ref{lem_monotonicity}%
\iffullversion 
 and recall Fig.~\ref{fig_subgradients_dual}%
\fi
). Mixing the associated policies $\pi_{\underline{\lambda}}$ and $\pi_{\overline{\lambda}}$ as described in \eqref{eq_stochastic_policies_ratios} yields a mixed policy, whose suboptimality to the solution of Problem \eqref{eq_problemFormulation_CCOC} converges to zero at least exponentially over the number of bisection steps.

\begin{theorem}[Suboptimality Bound, \cite{pfeiffer}]\label{thm_suboptimality_bounds2} Let the bisection algorithm be initialized with $[\underline{\lambda}_{\text{init}}, \overline{\lambda}_{\text{init}}]$. After $M$ bisection steps, the suboptimality of the policy constructed in \eqref{eq_stochastic_policies_ratios} compared to the solution $\pi^{\star}$ of Problem \eqref{eq_problemFormulation_CCOC} is bounded by
\begin{align*}
    C_0^{\pi_{\text{mix}}}(\Tilde{x}_0)\!-\!C_0^{\pi^{\star}}(\Tilde{x}_0)\!&\leq\!p_{\overline{\lambda}}(1\!-\!p_{\overline{\lambda}})(\overline{\lambda}\!-\!\underline{\lambda})(V_0^{\pi_{\overline{\lambda}}}(\Tilde{x}_0)\!-\!V_0^{\pi_{\underline{\lambda}}}(\Tilde{x}_0))\\&\leq 0.25 \left(\frac{1}{2}\right)^M(\overline{\lambda}_{\text{init}}-\underline{\lambda}_{\text{init}}).
\end{align*} 
\end{theorem} 
\smallskip

The process is summarized in Algorithm \ref{alg_final}, which runs the bisection algorithm until a prescribed suboptimality gap $\Delta$ is attained, which is guaranteed to be achieved in finitely many iterations by Theorem \ref{thm_suboptimality_bounds2}. The computational performance is thus mainly determined by the inner loop, which is solved using DP. For continuous state and action spaces, the value function has to be approximated, e.g., by using gridding or basis functions \cite{Bertsekas_1,farias_1,schmid2024joint}. 

\iffullversion
\begin{algorithm}
\caption{Joint Chance Constr. Optimal Control}\label{alg:two}
\KwData{$\mathcal{X},\mathcal{U},T,N,\tilde{x}_0,\alpha, \Delta$}
\texttt{\\}
\tcc{Check border cases}
$\overline{V}_0(\Tilde{x}_0) \gets $ Maximum safety recursion \eqref{eq_max_invariance_evaluation}\; 
$\underline{V}_0(\Tilde{x}_0)\gets $ Safety of minimum cost policy, \eqref{eq_max_invariance_evaluation} subject to inputs optimal for \eqref{eq_max_cost_evaluation}\;
\If{$\overline{V}_0(\Tilde{x}_0)\leq \alpha$ or $\underline{V}_0(\Tilde{x}_0)\geq \alpha$}{
      Use method in Proposition \ref{pro_Border_cases_method}
} 
\texttt{\\}
\tcc{Run recursion}
$\underline{\lambda}\leftarrow 0$, $\lambda \leftarrow  \overline{\lambda} \leftarrow$ Proposition \ref{thm_upper_lower_bound_on_lambda} \;
\While{true}{
    $p_{\overline{\lambda}} \leftarrow$  Equation \eqref{eq_stochastic_policies_ratios} \;
    $\delta\leftarrow p_{\overline{\lambda}}(1-p_{\overline{\lambda}})(\overline{\lambda}-\underline{\lambda})(V_0^{\pi_{\overline{\lambda}}}(\Tilde{x}_0) - V_0^{\pi_{\underline{\lambda}}}(\Tilde{x}_0))$ \;
    \If{$\delta \leq \Delta$}{
        return $p_{\overline{\lambda}},\pi_{\underline{\lambda}},\pi_{\overline{\lambda}}$ \;
    }
    $V_0^{\pi_{\lambda}}(\Tilde{x}_0), \pi_{\lambda} \gets $ Theorem \ref{thm_our_Dp-recursion} \; 
     \eIf{$V_0^{\pi_{\lambda}}(\Tilde{x}_0)\leq \alpha$}{
        $\underline{\lambda} \leftarrow  \lambda$ \;
     }{
        $\overline{\lambda} \leftarrow  \lambda$ \; 
          
     }
     $\lambda\leftarrow \frac{1}{2}(\overline{\lambda} + \underline{\lambda})$ \;
}
\label{alg_final}
\end{algorithm}
\fi

\section{Numerical Example}
\label{sec_num_example}

We compare our algorithmic solution to the one proposed by \cite{Ono_2}. Following \cite{Ono_2}, the application of  Boole's inequality $\mathbb{P}_{x_0}^\pi(x_{0:N}\in \mathcal{A}) \geq 1-\sum_{k=0}^N \mathbb{P}_{x_0}^\pi(x_k \in \mathcal{A}^c)$ to Problem \eqref{eq_problemFormulation_CCOC} leads to the following Lagrangian dual
\begin{align}
\begin{split}
    \max_{\lambda \in \mathbb{R}_{\geq 0}} \inf_{\pi \in \Pi_{\text{dm}}} \quad  &\mathbb{E}_{x_{0}}^\pi\left[ \ell_N(x_N) + \sum_{k=0}^{N-1} \ell_k(x_k,u_k) \right] \\
    &+ \lambda\left(\sum_{k=0}^N \mathbb{P}(x_{k}\in \mathcal{A}^c)-1+\alpha\right),\end{split}
    \label{eq_dual_problem_for_ono_variant}
\end{align} 
which displays a stage-wise decomposition structure. This allows the direct application of DP to compute the infimum, by using the terminal and stage-cost 
\begin{align}
    \begin{split}
    \ell_{\lambda,N}(x_N) &=\ell_N(x_N) + \lambda \mathbb{1}_{\mathcal{A}^c}(x_N),\\
    \ell_{\lambda,k}(x_k,u_k) &= \ell_k(x_k,u_k) + \lambda\mathbb{1}_{\mathcal{A}^c}(x_k).
    \end{split}
    \label{eq_dp_recursion_ono}
\end{align}
We compare this approach to the one proposed here via two examples on unicycle trajectory planning and one example on fishery management. The code used to generate the results is available at \url{https://github.com/NiklasSchmidResearch/JCC_opt_control.git}. 

\smallskip 
\begin{example}[Unicycle Trajectory Planning] 
\label{ex_quadcopter1} 
We consider an Euler discretized unicycle model with constant speed of $3$, given by $x_{k+1} = x_k + 3\begin{bmatrix}\cos(u) & \sin(u)\end{bmatrix}^{\top} + w_k$, where $w_k=\mathcal{N}([0 \ 0]^{\top},\text{diag}(5,5))$ is an additive disturbance, $u\in[0,2\pi]$ is the input, and $x_0=\begin{bmatrix}13 & 13\end{bmatrix}^{\top}$. We restrict the state space to $x_k\in\mathcal{X}=[-25,25]\times[-25,25]$ and aim to minimize the cost $\ell_k(x_k,u_k) = x_k^{\top}x_k,$ for $k \in [20]$, while staying in the set $\mathcal{A}$ with a predefined probability.


\ifsingleCol
    \begin{figure}[tbp!]
    \centering
        \begin{subfigure}[t]{0.49\columnwidth}
        \centering
        \includegraphics[width=.7\columnwidth]{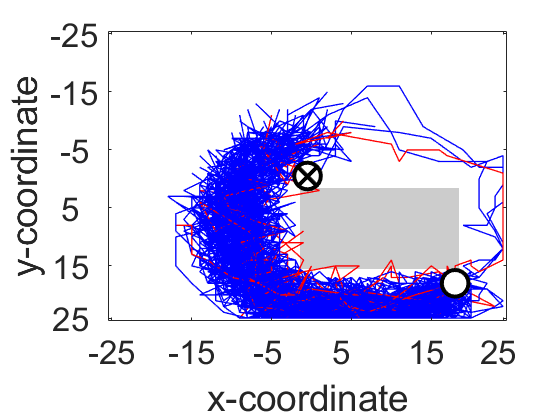}
        
        \end{subfigure}%
        \hfill
        \begin{subfigure}[t]{0.49\columnwidth}
        
        \centering
        \includegraphics[width=0.7\columnwidth]{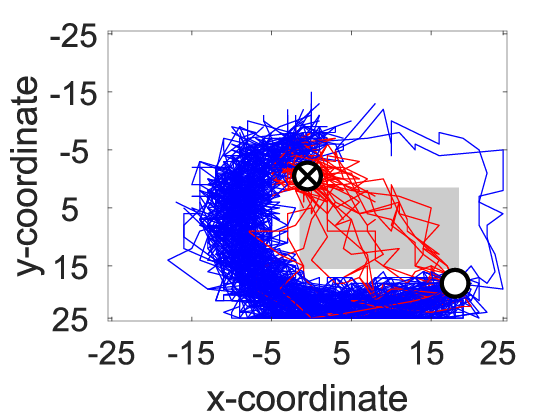}
        \end{subfigure}%
        \caption{Trajectories of $150$ Monte Carlo simulations of the unicycle in Example \ref{ex_quadcopter1}. The white area denotes the safe set $\mathcal{A}$, while the grey region should be avoided. The black circle denotes the initial state, the black circle with the cross the origin, which serves as the target state with zero stage cost. Blue curves denote safe trajectories, while red curves denote unsafe trajectories. The left plot shows the results for the policy of \cite{Ono_2}, the right plot the policy obtained by the recursion in Theorem \ref{thm_our_Dp-recursion}. Despite the different behaviour, both policies provide the same safety probability of $0.9$; the one on the right, however, attains a lower average cost. }
        \label{fig_trajectoryPlotsWithBox}
    \end{figure}
\else
    \ifNonArxivVersion
        \begin{figure}[tbp!]
        \centering
            \begin{subfigure}[t]{0.49\columnwidth}
            \centering
            \includegraphics[width=1\columnwidth]{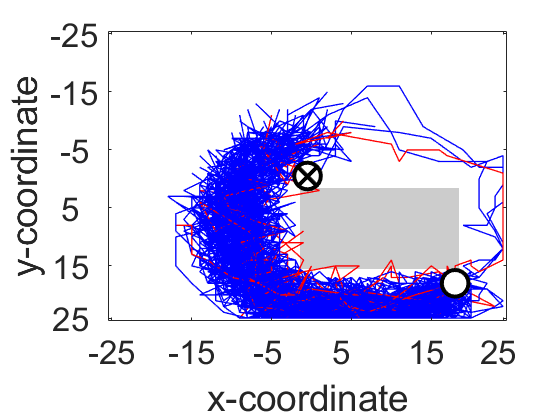}
            
            \end{subfigure}%
            \hfill
            \begin{subfigure}[t]{0.49\columnwidth}
            
            \centering
            \includegraphics[width=1\columnwidth]{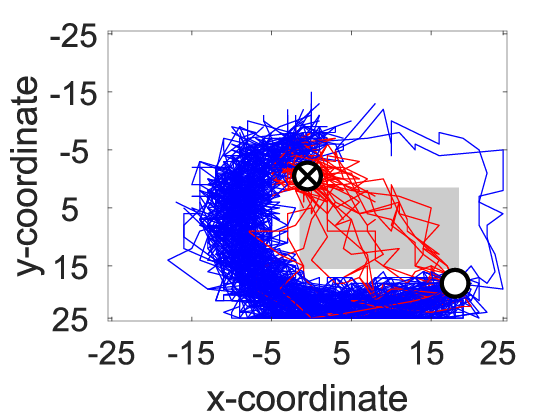}
            \end{subfigure}%
            \caption{Trajectories of $150$ simulations of the unicycle in Example \ref{ex_quadcopter1}. The white area is the safe set $\mathcal{A}$, grey unsafe. The black circle is the initial state, the black crossed circle the origin. Blue curves are safe trajectories, red curves unsafe. The left plot shows the policy of \cite{Ono_2}, the right plot shows the policy from Theorem \ref{thm_our_Dp-recursion}. Both have a safety of 0.9, but the right one has a lower average cost.
            }
            \label{fig_trajectoryPlotsWithBox}
        \end{figure}
    \else
        \begin{figure}[tbp!]
            \centering
            \includegraphics[height=3.1cm]{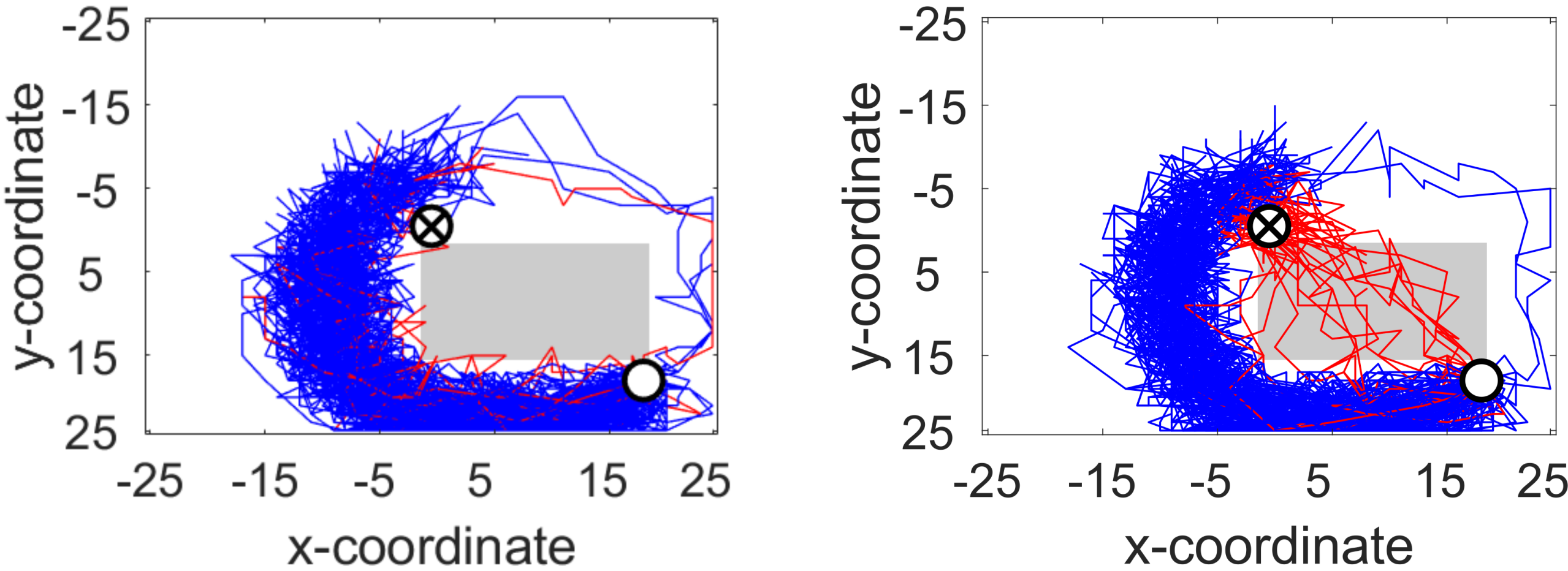}
            \caption{Trajectories of $150$ Monte Carlo simulations of the unicycle in Example \ref{ex_quadcopter1}. The white area denotes the safe set $\mathcal{A}$, while the grey region should be avoided. The black circle denotes the initial state, the black circle with the cross the origin, which serves as the target state with zero stage cost. Blue curves denote safe trajectories, while red curves denote unsafe trajectories. The left plot shows the results for the policy of \cite{Ono_2}, the right plot the policy obtained by the recursion in Theorem \ref{thm_our_Dp-recursion}. Despite the different behaviour, both policies provide the same safety probability of $0.9$; the one on the right, however, attains a lower average cost. }
            \label{fig_trajectoryPlotsWithBox}
        \end{figure}
    \fi
\fi


\rev{Note that the state and action spaces of the system are continuous, which causes the DP recursions \eqref{eq_our_dp_recursion_innerdual} and \eqref{eq_dp_recursion_ono} to be infinite dimensional problems. To enable computation, we grid the state space into $50\times50$ discrete states and the action space into $8$ discrete inputs following the procedure described in \cite{abate2010approximate}.} \rev{We then compare the policies generated by recursions \eqref{eq_dp_recursion_ono} and \eqref{eq_our_dp_recursion_innerdual}. Note that the safety probability of the policies resulting from both recursions scale monotonically with $\lambda$. Similar to Algorithm \ref{alg_final}, we run a bisection on $\lambda$ to achieve policies with a safety probability of $0.9$ for both, \eqref{eq_dp_recursion_ono} and \eqref{eq_our_dp_recursion_innerdual}. The bisection gives rise to $\lambda=8205$ and $\lambda=12645$, respectively.} 
\iffullversion
The resulting policies are shown in Fig.~\ref{fig_policyPlotsWithBox}. \fi Indeed, over $150$ Monte Carlo runs (Fig.~\ref{fig_trajectoryPlotsWithBox}) both policies achieve a safety probability of $0.9$, however the policy developed in this work achieves a lower average cost of $7260$ vs. $7580$. This is because the policy from \cite{Ono_2} tends to leave the restricted area as fast as possible after entering it. On the other hand, our policy keeps transitioning through the restricted area once it enters, since as soon as the trajectory is tagged as unsafe the policy only aims to minimize the expected cost. This difference is due to the structure of the DP recursions employed by the two algorithms. In our DP recursion (see Theorem \ref{th_recursion}) we enforce a one-time penalty of $\lambda$ if the trajectory runs unsafe, independent of the duration it remains unsafe, while in \eqref{eq_dp_recursion_ono} the policy incurs a penalty of $\lambda$ for every time-step that the state is outside the safe set. \rev{Similar incentive violations have been observed in reinforcement learning, where an initially myopic agent may exploit short-term rewards unaware of jeopardizing its long-term performance \cite{bonyadi2022self}. This effect disappears after sufficient training once the optimal policy is learned. In constrained MDPs as ours, however, even the optimal policy might occasionally require violating safety constraints intentionally.}

\iffullversion
\begin{figure*}
\centering
    \begin{tikzpicture}
        \node[inner sep=0pt] (russell) at (-5,0)
        {\includegraphics[trim=6cm 10cm 5cm 10cm, clip,width=4cm]{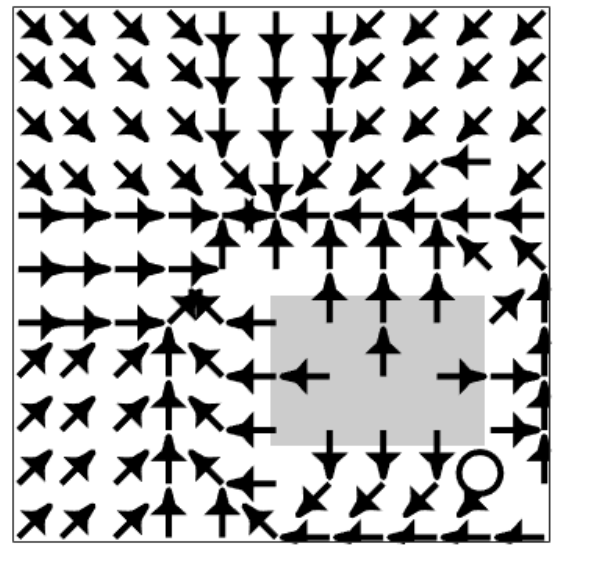}};
        \node[inner sep=0pt] (russell) at (1,0)
        {\includegraphics[trim=6cm 10cm 5cm 10cm, clip,width=4cm]{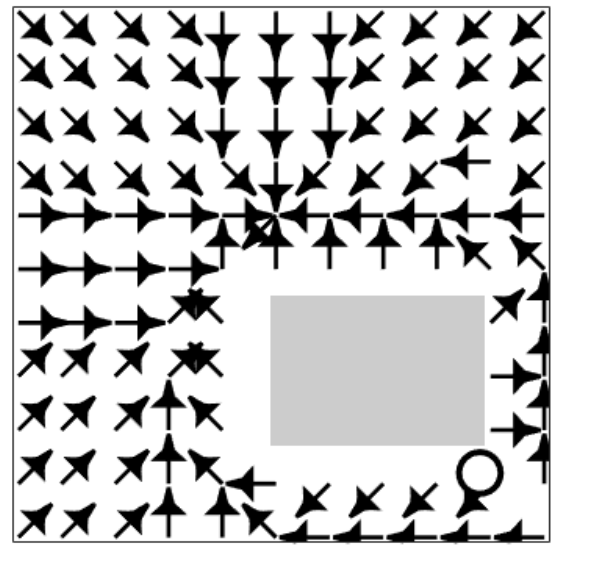}};
        \node[inner sep=0pt] (russell) at (5,0)
        {\includegraphics[trim=6cm 10cm 5cm 10cm, clip,width=4cm]{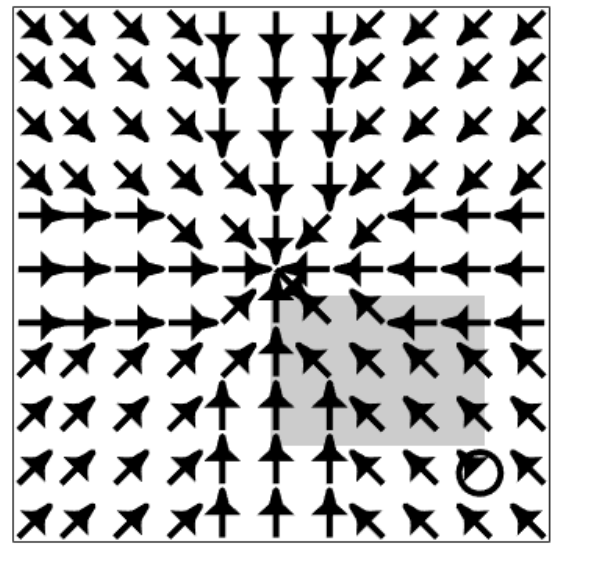}};
        \draw (-.4,1.72) node [fill=white,inner sep=1pt]{$b_k=1$};
        \draw (3.6,1.72) node [fill=white,inner sep=1pt]{$b_k=0$};
        \draw (-4.95,-2.2) node [fill=white,inner sep=1pt]{Approach by \cite{Ono_2}};
        \draw (3.1,-2.2) node [fill=white,inner sep=1pt]{Our approach};
    \end{tikzpicture}
    \caption{Optimal inputs  at time-step $k=0$ associated to the recursion in \eqref{eq_dp_recursion_ono} (left) and $\pi_{\overline{\lambda}}$ (middle for $b_k=1$ and right for $b_k=0$) for Example \ref{ex_quadcopter1}. Safe states are marked white, unsafe states grey. The initial state is marked with a black circle. }
    \label{fig_policyPlotsWithBox}
\end{figure*}
\fi

\end{example}

The unsafe set is now replaced in the middle of the state space so that the cheapest states are unsafe. The initial state is $x_0=\begin{bmatrix}19 & 19\end{bmatrix}^{\top}$ (see the left plot in Fig.~\ref{fig_pareto_front_plot}). To compare our DP recursion in Theorem \ref{thm_our_Dp-recursion} with that in \cite{Ono_2} we evaluate the Pareto fronts of the two recursions by varying $\lambda\in[100,10^6]$ (see the right plot in Fig.~\ref{fig_pareto_front_plot}). We compare 
\begin{itemize}
    \item (\color{niceGreen}Dashed\color{black}): Evaluating the policy via Theorem \ref{thm_our_Dp-recursion} and its associated safety via \eqref{eq_max_invariance_evaluation} (our approach).
    \item (\color{blue}Dash-Dotted\color{black}): Evaluating the policy via  \eqref{eq_dp_recursion_ono} and its associated safety via \eqref{eq_max_invariance_evaluation}. Compared to the dashed graph, this shows the conservativeness introduced by Boole's inequality in the DP recursion \eqref{eq_dp_recursion_ono}.
    \item (\color{red}Dotted\color{black}) Evaluating the policy via \eqref{eq_dp_recursion_ono} and its associated safety via Boole's inequality \cite{Ono_2}. Comparing the dotted with the dash-dotted graph yields the conservatism of Boole's inquality for evaluating the safety of a policy, comparing the dotted with the dashed graph the performance difference between the approach by \cite{Ono_2} and ours.
\end{itemize}


Clearly, the dashed graph dominates the dash-dotted graph, which in turn dominates the dotted graph, showing the effectiveness of our approach in reducing conservatism. For $\lambda=0$, all methods are expected to generate policies that achieve the minimum achievable cost, as constraint violations are not penalized. This can be seen by the dashed and dash-dotted graph converging for small probabilities (the dotted graph achieves this cost at negative safety probabilities). 

Interestingly, the Pareto fronts might not necessarily converge for high values of $\lambda$%
\iffullversion{. Consider the fictitious example with two policies $\pi_1, \pi_2$, where $\pi_1$ is the safest policy and has a probability of $0.1$ to become unsafe at the first time-step and zero otherwise, and $\pi_2$ has a probability of $0.2$ to be unsafe at the very last time-step and zero for all other time-steps. Let $N=3$. Even if $C_0^{\pi_1}(\Tilde{x}_0)=C_0^{\pi_2}(\Tilde{x}_0)$, the DP recursion in \cite{Ono_2} will prefer policy $\pi_2$ for any $\lambda>0$ since the incurred penalty will be $0.3\lambda$ and $0.2\lambda$ for policy $\pi_1$ and $\pi_2$, respectively. Thus, although $\pi_1$ might be the safest policy, it will never be considered optimal in \cite{Ono_2}, even for arbitrarily large $\lambda$. In our example, this can be observed by the graphs approaching, but never really touching each other for higher safety values. }
\fi 
\iffullversion
The only exception is when a safety of one is attainable since the conservatism of Boole's inequality converges to zero as the safety approaches one. However, if we aim for lower requirements on safety, or if the highest achievable safety is far enough from one, the conservatism becomes significant. This can especially be observed for small safety values, e.g., to guarantee a safety of $0.1$, the approach by \cite{Ono_2} proposes a policy with cost $4113$, while the DP recursion \eqref{eq_dp_recursion_ono} is able to compute a policy with cost $3330$ and our algorithm finds a policy with cost $1522$. 
\fi

We also tested Algorithm \ref{alg_final} using $\alpha=0.6$ and a suboptimality bound of $\Delta=10^{-6}$, achieving $\Delta$ after $20$ iterations. For the approach in \cite{Ono_2} we ran a fixed number of $30$ iterations. Our approach returned a policy with cost $4122$ compared to $5822$ using the approach in \cite{Ono_2}. 

\ifsingleCol
    \begin{figure}[!tbh]
        \centering    
        \includegraphics[height=5cm]{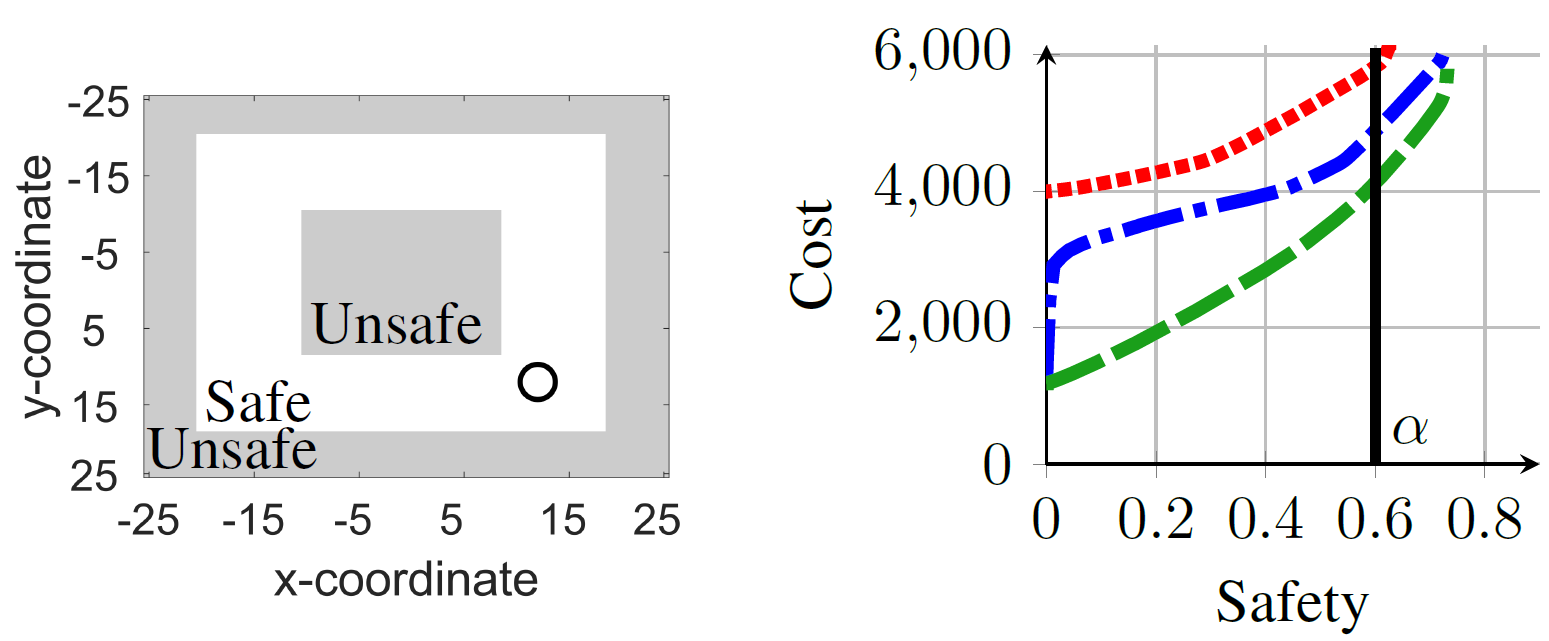}
        \caption{The left plot shows the safe (white) and unsafe (grey) regions of the state space and the initial state (black circle) in example \ref{ex_quadcopter_2}. The right plot shows the Pareto fronts when using our algorithm (\color{niceGreen}Dashed\color{black}), our algorithm but replacing the DP recursion with \eqref{eq_dp_recursion_ono} (\color{blue}Dash-dotted\color{black}), and when using the full approach in \cite{Ono_2} (\color{red}Dotted\color{black}). In fact, the dotted graph continuous to negative probabilities due to convervativeness of Boole's inequality. }
        \label{fig_pareto_front_plot}
    \end{figure}
\else
    \ifNonArxivVersion
        \begin{figure}[!tbh]
            \centering    
            \begin{subfigure}[c]{0.4\columnwidth} 
            \begin{tikzpicture}
                \node[inner sep=0pt] (russell) at (0,0) {
                \includegraphics[width=\columnwidth]{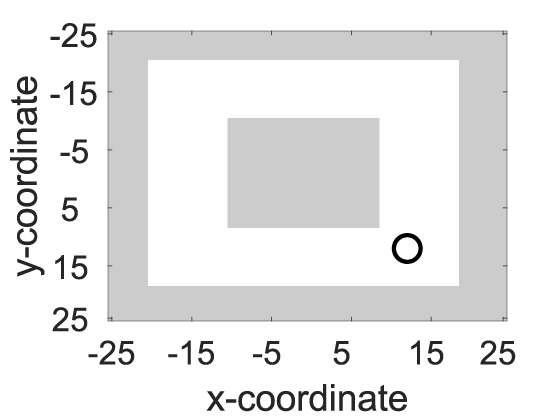}     
                };
                \draw (0.15, .05) node {Unsafe};
                \draw (-.5, -.3) node {Safe};
                \draw (-.6, -.57) node {Unsafe};
            \end{tikzpicture}
            \end{subfigure}
            \hfill
            \begin{subfigure}[c]{0.5\columnwidth}    
            \begin{tikzpicture}
                \begin{axis}[height=3.2 cm,width=\columnwidth,ymin=0,xmin=0,xmax=0.9,axis lines=left,grid=major,xlabel={Safety}, ylabel={Cost}]
                    \addplot +[blue,line width=0.8mm, mark=none,dash pattern={on 10pt off 2pt on 2pt off 2pt}] coordinates {(0.0000,1233.0710)(0.0000,1272.8835)(0.0000,1311.3944)(0.0003,1450.4191)(0.0051,2406.1573)(0.0130,2917.4791)(0.0281,3049.5090)(0.0313,3070.8899)(0.0415,3135.7221)(0.0629,3220.4099)(0.0940,3316.8170)(0.1421,3426.1262)(0.1576,3466.2892)(0.1655,3483.2946)(0.1708,3499.6366)(0.1910,3545.1290)(0.2254,3620.1778)(0.2771,3723.3211)(0.3260,3814.2605)(0.3667,3885.1327)(0.3742,3900.3404)(0.3793,3910.9159)(0.4152,3982.4221)(0.4178,3988.5083)(0.4325,4022.9898)(0.4347,4029.1623)(0.4410,4046.6957)(0.4424,4051.0027)(0.4462,4063.7663)(0.4553,4095.9725)(0.4859,4209.1153)(0.5064,4289.0253)(0.5307,4388.1575)(0.5349,4407.5925)(0.5385,4425.0266)(0.5417,4442.5795)(0.5448,4460.6863)(0.5473,4476.6480)(0.5503,4497.2008)(0.5603,4572.3413)(0.5770,4700.2196)(0.6300,5130.3751)(0.6918,5666.0464)(0.7031,5769.9766)(0.7161,5881.7966)(0.7163,5883.0683)(0.7164,5885.6479)(0.7170,5892.9705)(0.7178,5902.2202)(0.7180,5906.0590)(0.7194,5922.6046)(0.7200,5932.6828)(0.7201,5933.3858)(0.7202,5935.7580)(0.7202,5935.8609)(0.7206,5943.6004)(0.7210,5953.0010)(0.7254,6044.5395)(0.7272,6082.6391)(0.7275,6091.0061)(0.7291,6132.2185)(0.7291,6132.7110)(0.7291,6133.7411)(0.7291,6133.8689)(0.7291,6134.0322)(0.7291,6134.0334)(0.7291,6134.0348)(0.7291,6134.0349)(0.7291,6134.1492)(0.7291,6134.4550)(0.7291,6134.5107)(0.7291,6135.4899)(0.7291,6135.7577)(0.7291,6135.8664)(0.7291,6136.1223)(0.7291,6136.1263)(0.7291,6136.2292)(0.7291,6136.6871)(0.7291,6136.6897)(0.7291,6136.6899)(0.7292,6138.7205)(0.7292,6139.1028)(0.7292,6139.1035)(0.7292,6140.9453)(0.7292,6140.9457)(0.7292,6140.9459)(0.7292,6140.9459)(0.7292,6140.9461)(0.7292,6140.9461)(0.7292,6140.9465)(0.7292,6140.9594)(0.7292,6140.9594)(0.7292,6140.9594)(0.7292,6140.9594)(0.7292,6141.3683)(0.7292,6141.3683)(0.7292,6141.3685)(0.7292,6141.3685)(0.7292,6141.3685)(0.7292,6141.3685)};
                    \addplot +[line width=0.8mm, mark=none,niceGreen,dash pattern={on 10pt off 2pt}] coordinates {(0.0000,1185.9964)(0.0000,1185.9964)(0.0000,1185.9965)(0.0000,1185.9965)(0.0000,1185.9965)(0.0000,1185.9965)(0.0000,1185.9966)(0.0000,1185.9966)(0.0000,1185.9966)(0.0000,1185.9967)(0.0000,1185.9968)(0.0000,1185.9971)(0.0000,1185.9973)(0.0000,1185.9975)(0.0000,1185.9977)(0.0000,1185.9980)(0.0000,1185.9987)(0.0000,1185.9997)(0.0000,1186.0013)(0.0000,1186.0021)(0.0000,1186.0047)(0.0000,1186.0100)(0.0000,1186.0136)(0.0000,1186.0205)(0.0001,1186.0388)(0.0001,1186.0518)(0.0001,1186.0864)(0.0002,1186.1371)(0.0002,1186.2255)(0.0003,1186.3728)(0.0005,1186.6433)(0.0008,1187.1092)(0.0014,1188.2520)(0.0021,1189.7998)(0.0036,1193.2356)(0.0065,1200.2187)(0.0128,1217.6433)(0.0245,1252.8175)(0.0474,1328.5788)(0.0902,1484.1582)(0.1637,1775.4582)(0.2724,2245.1160)(0.3958,2826.9517)(0.4488,3103.0551)(0.4765,3261.0908)(0.5284,3589.1339)(0.5550,3772.3517)(0.5834,3987.0008)(0.6127,4231.3261)(0.6483,4556.8143)(0.6868,4939.6100)(0.7085,5176.2521)(0.7148,5250.9621)(0.7180,5292.5027)(0.7189,5305.7662)(0.7197,5318.6884)(0.7207,5336.4006)(0.7220,5360.6431)(0.7231,5383.9020)(0.7246,5417.7375)(0.7263,5461.6095)(0.7281,5512.5636)(0.7288,5532.0729)(0.7291,5543.8091)(0.7291,5544.4333)(0.7291,5544.5873)(0.7291,5545.2712)(0.7292,5547.9187)(0.7292,5549.2226)(0.7293,5551.2386)(0.7293,5555.7077)(0.7294,5563.4681)(0.7296,5574.9398)(0.7297,5587.8304)(0.7299,5607.1472)(0.7313,5742.4421)(0.7313,5742.4824)(0.7315,5765.7902)(0.7317,5786.0301)(0.7317,5787.8644)(0.7317,5787.9062)(0.7317,5787.9607)(0.7317,5787.9694)(0.7317,5787.9831)(0.7317,5788.0121)(0.7317,5788.0881)(0.7317,5788.1484)(0.7317,5788.1485)(0.7317,5788.1485)(0.7317,5788.1522)(0.7317,5788.1850)(0.7317,5788.1850)(0.7317,5788.1850)(0.7317,5788.5073)(0.7317,5793.6501)(0.7317,5793.9398)(0.7317,5793.9398)(0.7317,5793.9398)(0.7317,5794.2202)(0.7317,5794.2202)};
                    \addplot +[line width=0.8mm, mark=none,red,densely dotted] coordinates {(-0.0156,3982.4221)(-0.0081,3988.5083)(0.0306,4022.9898)(0.0369,4029.1623)(0.0532,4046.6957)(0.0569,4051.0027)(0.0667,4063.7663)(0.0890,4095.9725)(0.1612,4209.1153)(0.2078,4289.0253)(0.2613,4388.1575)(0.2708,4407.5925)(0.2787,4425.0266)(0.2859,4442.5795)(0.2925,4460.6863)(0.2979,4476.6480)(0.3041,4497.2008)(0.3247,4572.3413)(0.3571,4700.2196)(0.4557,5130.3751)(0.5699,5666.0464)(0.5901,5769.9766)(0.6099,5881.7966)(0.6101,5883.0683)(0.6104,5885.6479)(0.6114,5892.9705)(0.6125,5902.2202)(0.6129,5906.0590)(0.6146,5922.6046)(0.6155,5932.6828)(0.6156,5933.3858)(0.6158,5935.7580)(0.6158,5935.8609)(0.6163,5943.6004)(0.6168,5953.0010)(0.6214,6044.5395)(0.6233,6082.6391)(0.6236,6091.0061)(0.6252,6132.2185)(0.6253,6132.7110)(0.6253,6133.7411)(0.6253,6133.8689)(0.6253,6134.0322)(0.6253,6134.0334)(0.6253,6134.0348)(0.6253,6134.0349)(0.6253,6134.1492)(0.6253,6134.4550)(0.6253,6134.5107)(0.6253,6135.4899)(0.6253,6135.7577)(0.6253,6135.8664)(0.6253,6136.1223)(0.6253,6136.1263)(0.6253,6136.2292)(0.6253,6136.6871)(0.6253,6136.6897)(0.6253,6136.6899)(0.6253,6138.7205)(0.6253,6139.1028)(0.6253,6139.1035)(0.6254,6140.9453)(0.6254,6140.9457)(0.6254,6140.9459)(0.6254,6140.9459)(0.6254,6140.9461)(0.6254,6140.9461)(0.6254,6140.9465)(0.6254,6140.9594)(0.6254,6140.9594)(0.6254,6140.9594)(0.6254,6140.9594)(0.6254,6141.3683)(0.6254,6141.3683)(0.6254,6141.3685)(0.6254,6141.3685)(0.6254,6141.3685)(0.6254,6141.3685)};
                    \addplot +[line width=.6mm, mark=none,black] coordinates {(0.6,0)(0.6,6100)};
                \end{axis}
                \draw (2.1, 0.2) node {\color{black}$\alpha$\color{black}};
            \end{tikzpicture}
            \end{subfigure}
            \caption{The left plot shows the safe (white) and unsafe (grey) states and the initial state (black circle) in the second unicycle example. The right plot shows the Pareto fronts when using our algorithm (\color{niceGreen}Dashed\color{black}), our algorithm, replacing the DP recursion with \eqref{eq_dp_recursion_ono} (\color{blue}Dash-dotted\color{black}), and the approach in \cite{Ono_2} (\color{red}Dotted\color{black}). The dotted graph continuous to negative probabilities due to convervativeness of Boole's inequality (not plotted). }
            \label{fig_pareto_front_plot}
        \end{figure}
    \else
        \begin{figure}[!tbh]
            \centering    
            \includegraphics[height=3.5cm]{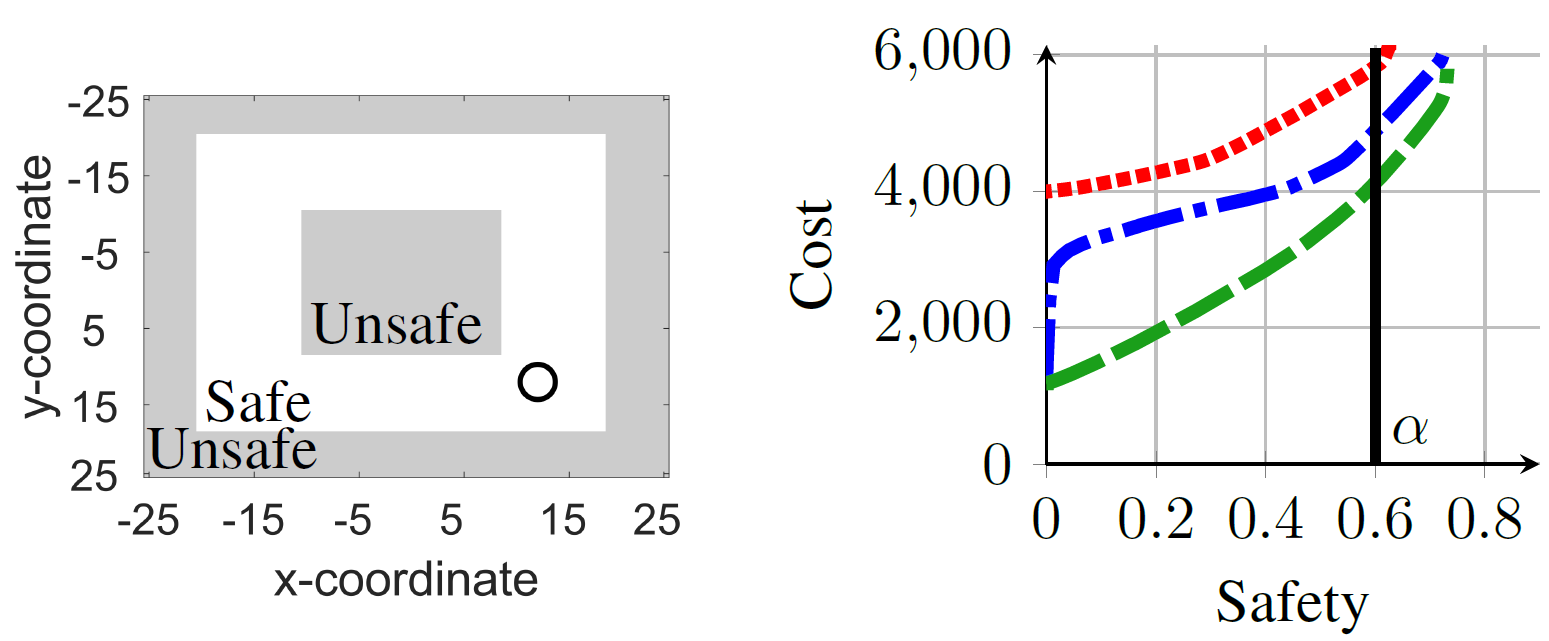}
            \caption{The left plot shows the safe (white) and unsafe (grey) regions of the state space and the initial state (black circle) in example \ref{ex_quadcopter_2}. The right plot shows the Pareto fronts when using our algorithm (\color{niceGreen}Dashed\color{black}), our algorithm but replacing the DP recursion with \eqref{eq_dp_recursion_ono} (\color{blue}Dash-dotted\color{black}), and when using the full approach in \cite{Ono_2} (\color{red}Dotted\color{black}). In fact, the dotted graph continuous to negative probabilities due to convervativeness of Boole's inequality. }
            \label{fig_pareto_front_plot}
        \end{figure}
    \fi
\fi

\smallskip
\begin{example}[Fisheries Management]
\label{ex_blubbblubb}
We adapt the fisheries management example from \cite{summers_1} using the model from \cite{pitchford_1}. The evolution of fish biomass $x_k$ in a reservoir is given by
\begin{align*}
    x_{k+1}=(1-v_k)x_k + \gamma_k R(x_k) - F(x_k,u_k),
\end{align*}
where $R(\cdot)$, $F(\cdot,\cdot)$ are functions representing recruitment and catch, $v_k\sim \mathcal{N}(0.2,0.01)$ the natural mortality rate and $\gamma_k\sim \mathcal{N}(1,0.36)$ variability in the recruitment. The catch function is imposed by the government and described by 
\begin{align*}
    F(x_k,u_k) = \max\left\{\delta_ku_kC, 
 \ \delta_ku_kC\frac{x_k}{L}\right\},
\end{align*}
where $\delta_k\sim \mathcal{N}(1.1,0.04)$ is a variability in the catch with $\delta_k>0$, $u_k\in[0,1]$ our input variable denoting catch effort, $F$ the maximum catch, and $L$ the biomass limit of the reservoir. The recruitment function is described by
\begin{align*}
    R(x_k) = x_k\left(1-\frac{x_k}{L}\right)\text{sgm}\left(\frac{x_k-\mu}{\sigma^2}\right),
\end{align*}
where the term $(1-x_k/L)$ models decline in recruitment with saturation of the reservoir. The sigmoid $\text{sgm}(x)=\frac{1}{1+e^{-x}}$ has been added here to introduce a bifurcation in the dynamics and could model a rapid decline in recruitment when the population becomes too small. 

We use $L=40$, $M=10$, $\mu=20$, $\sigma=5$. Empirically, using Monte Carlo simulations, we found that the fish population is at almost zero probability of recovering whenever the biomass surpasses $13$ units, even if we stop fishing. Our goal is to catch as much fish as possible within $100$ time-steps, while guaranteeing a probability of at least $\alpha=0.75$ to preserve at least $13$ units of biomass throughout. We define $\mathcal{X}=[0,60]$, discretize the state space into $60$ discrete states, the input space $\mathcal{U}=[0,1]$ into $6$ discrete inputs and use $\Delta=10^{-6}$. 

Due to the conservatism of Boole's inequality, a solution to this problem would be infeasible in \cite{Ono_2}. However, to allow for a comparison with our results, we again run our algorithm, replacing our DP recursion with \eqref{eq_dp_recursion_ono}. The policies computed this way and with our method are depicted in Fig.~\ref{fig_fish_policies}.

\begin{figure*}[t!]
    \centering
    \begin{tikzpicture}
        \draw (2.5, .25) node {$k=0$};
        \begin{axis}[width=0.25\textwidth,height=3.4cm,%
                    ylabel={Catch Effort}, xlabel={Biomass $x_0$}, ylabel shift = -.1cm, xlabel shift = -.1cm, axis x line = bottom,axis y line = left]
        \addplot +[mark=none,olive, very thick] coordinates {(13.0000,1.0000)(14.0000,1.0000)(15.0000,0.0000)(16.0000,0.0000)(17.0000,0.0000)(18.0000,0.0000)(19.0000,0.0000)(20.0000,0.0000)(21.0000,0.0000)(22.0000,0.0000)(23.0000,0.0000)(24.0000,0.0000)(25.0000,0.0000)(26.0000,0.0000)(27.0000,0.0000)(28.0000,0.2000)(29.0000,0.4000)(30.0000,0.4000)(31.0000,0.4000)(32.0000,0.4000)(33.0000,0.6000)(34.0000,0.6000)(35.0000,0.8000)(36.0000,0.8000)(37.0000,0.6000)(38.0000,0.6000)(39.0000,0.8000)(40.0000,0.8000)(41.0000,0.6000)(42.0000,0.8000)(43.0000,0.8000)(44.0000,1.0000)(45.0000,0.8000)(46.0000,1.0000)(47.0000,1.0000)(48.0000,1.0000)(49.0000,1.0000)(50.0000,1.0000)(51.0000,1.0000)(52.0000,1.0000)(53.0000,1.0000)(54.0000,1.0000)(55.0000,1.0000)(56.0000,1.0000)(57.0000,1.0000)(58.0000,1.0000)(59.0000,1.0000)(60.0000,1.0000)};
        \addplot +[mark=none,blue, densely dotted, very thick] coordinates {(13.0000,1.0000)(14.0000,0.0000)(15.0000,0.0000)(16.0000,0.0000)(17.0000,0.0000)(18.0000,0.0000)(19.0000,0.0000)(20.0000,0.0000)(21.0000,0.0000)(22.0000,0.0000)(23.0000,0.0000)(24.0000,0.0000)(25.0000,0.0000)(26.0000,0.0000)(27.0000,0.0000)(28.0000,0.0000)(29.0000,0.2000)(30.0000,0.2000)(31.0000,0.2000)(32.0000,0.4000)(33.0000,0.4000)(34.0000,0.6000)(35.0000,0.4000)(36.0000,0.6000)(37.0000,0.6000)(38.0000,0.6000)(39.0000,0.6000)(40.0000,0.6000)(41.0000,0.6000)(42.0000,0.6000)(43.0000,0.8000)(44.0000,0.8000)(45.0000,0.8000)(46.0000,1.0000)(47.0000,1.0000)(48.0000,1.0000)(49.0000,1.0000)(50.0000,1.0000)(51.0000,1.0000)(52.0000,1.0000)(53.0000,1.0000)(54.0000,1.0000)(55.0000,1.0000)(56.0000,1.0000)(57.0000,1.0000)(58.0000,1.0000)(59.0000,1.0000)(60.0000,1.0000)};
        \end{axis}
    \end{tikzpicture}
    \hfill
    \begin{tikzpicture}
        \draw (2.5, .25) node {$k=49$};
        \begin{axis}[width=0.25\textwidth,height=3.4cm,%
                    ylabel={Catch Effort}, xlabel={Biomass $x_{49}$}, ylabel shift = -.1cm, xlabel shift = -.1cm, axis x line = bottom,axis y line = left]
        \addplot +[mark=none,olive, very thick] coordinates {(13.0000,1.0000)(14.0000,1.0000)(15.0000,1.0000)(16.0000,0.0000)(17.0000,0.0000)(18.0000,0.0000)(19.0000,0.0000)(20.0000,0.0000)(21.0000,0.0000)(22.0000,0.0000)(23.0000,0.0000)(24.0000,0.0000)(25.0000,0.0000)(26.0000,0.0000)(27.0000,0.2000)(28.0000,0.2000)(29.0000,0.4000)(30.0000,0.4000)(31.0000,0.4000)(32.0000,0.4000)(33.0000,0.6000)(34.0000,0.6000)(35.0000,0.8000)(36.0000,0.8000)(37.0000,0.6000)(38.0000,0.8000)(39.0000,0.8000)(40.0000,0.8000)(41.0000,0.6000)(42.0000,0.8000)(43.0000,0.8000)(44.0000,1.0000)(45.0000,0.8000)(46.0000,1.0000)(47.0000,1.0000)(48.0000,1.0000)(49.0000,1.0000)(50.0000,1.0000)(51.0000,1.0000)(52.0000,1.0000)(53.0000,1.0000)(54.0000,1.0000)(55.0000,1.0000)(56.0000,1.0000)(57.0000,1.0000)(58.0000,1.0000)(59.0000,1.0000)(60.0000,1.0000)};
        \addplot +[mark=none,blue, densely dotted, very thick] coordinates {(13.0000,1.0000)(14.0000,0.0000)(15.0000,0.0000)(16.0000,0.0000)(17.0000,0.0000)(18.0000,0.0000)(19.0000,0.0000)(20.0000,0.0000)(21.0000,0.0000)(22.0000,0.0000)(23.0000,0.0000)(24.0000,0.0000)(25.0000,0.0000)(26.0000,0.0000)(27.0000,0.0000)(28.0000,0.2000)(29.0000,0.4000)(30.0000,0.4000)(31.0000,0.4000)(32.0000,0.4000)(33.0000,0.4000)(34.0000,0.6000)(35.0000,0.4000)(36.0000,0.6000)(37.0000,0.6000)(38.0000,0.6000)(39.0000,0.6000)(40.0000,0.6000)(41.0000,0.6000)(42.0000,0.8000)(43.0000,0.8000)(44.0000,0.8000)(45.0000,0.8000)(46.0000,1.0000)(47.0000,1.0000)(48.0000,1.0000)(49.0000,1.0000)(50.0000,1.0000)(51.0000,1.0000)(52.0000,1.0000)(53.0000,1.0000)(54.0000,1.0000)(55.0000,1.0000)(56.0000,1.0000)(57.0000,1.0000)(58.0000,1.0000)(59.0000,1.0000)(60.0000,1.0000)};
        \end{axis}
    \end{tikzpicture}
    \hfill
    \begin{tikzpicture}
        \draw (2.5, .25) node {$k=99$};
        \begin{axis}[width=0.25\textwidth,height=3.4cm,%
                    ylabel={Catch Effort}, xlabel={Biomass $x_{99}$}, ylabel shift = -.1cm, xlabel shift = -.1cm, axis x line = bottom,axis y line = left,
                    legend style={at={(1.2,0.6)},anchor=west}]
        \addplot +[mark=none,olive, very thick] coordinates {(13.0000,0.0000)(14.0000,0.0000)(15.0000,0.0000)(16.0000,0.2000)(17.0000,0.4000)(18.0000,0.4000)(19.0000,0.8000)(20.0000,0.8000)(21.0000,1.0000)(22.0000,1.0000)(23.0000,1.0000)(24.0000,1.0000)(25.0000,1.0000)(26.0000,1.0000)(27.0000,1.0000)(28.0000,1.0000)(29.0000,1.0000)(30.0000,1.0000)(31.0000,1.0000)(32.0000,1.0000)(33.0000,1.0000)(34.0000,1.0000)(35.0000,1.0000)(36.0000,1.0000)(37.0000,1.0000)(38.0000,1.0000)(39.0000,1.0000)(40.0000,1.0000)(41.0000,1.0000)(42.0000,1.0000)(43.0000,1.0000)(44.0000,1.0000)(45.0000,1.0000)(46.0000,1.0000)(47.0000,1.0000)(48.0000,1.0000)(49.0000,1.0000)(50.0000,1.0000)(51.0000,1.0000)(52.0000,1.0000)(53.0000,1.0000)(54.0000,1.0000)(55.0000,1.0000)(56.0000,1.0000)(57.0000,1.0000)(58.0000,1.0000)(59.0000,1.0000)(60.0000,1.0000)};
        \addplot +[mark=none,blue, densely dotted, very thick] coordinates {(13.0000,1.0000)(14.0000,0.0000)(15.0000,0.2000)(16.0000,0.4000)(17.0000,0.6000)(18.0000,0.8000)(19.0000,1.0000)(20.0000,1.0000)(21.0000,1.0000)(22.0000,1.0000)(23.0000,1.0000)(24.0000,1.0000)(25.0000,1.0000)(26.0000,1.0000)(27.0000,1.0000)(28.0000,1.0000)(29.0000,1.0000)(30.0000,1.0000)(31.0000,1.0000)(32.0000,1.0000)(33.0000,1.0000)(34.0000,1.0000)(35.0000,1.0000)(36.0000,1.0000)(37.0000,1.0000)(38.0000,1.0000)(39.0000,1.0000)(40.0000,1.0000)(41.0000,1.0000)(42.0000,1.0000)(43.0000,1.0000)(44.0000,1.0000)(45.0000,1.0000)(46.0000,1.0000)(47.0000,1.0000)(48.0000,1.0000)(49.0000,1.0000)(50.0000,1.0000)(51.0000,1.0000)(52.0000,1.0000)(53.0000,1.0000)(54.0000,1.0000)(55.0000,1.0000)(56.0000,1.0000)(57.0000,1.0000)(58.0000,1.0000)(59.0000,1.0000)(60.0000,1.0000)};
        \addlegendentry{Our approach};
        \addlegendentry{Approach in \cite{Ono_2}}
        \end{axis}
    \end{tikzpicture}
    \caption{\rev{Optimal policies for Example \ref{ex_blubbblubb} at different time-steps using our approach (solid green) and that in \cite{Ono_2} (dotted blue)}.}
    \label{fig_fish_policies}
\end{figure*}

Intuitively, the less fish remains, the less aggressive the catch effort. Interestingly, our approach tends to fish more aggressively early on, while \cite{Ono_2} fishes more aggressively towards the end. Both aim for maximum catch with a risk of crashing the fish population. However, while we penalize violation of constraints once, the penalty is repeated and accumulated in recursion \eqref{eq_dp_recursion_ono} used by \cite{Ono_2}. Thus, our method exploits situations of high risk and tries to fish whatever remains, while the approach by \cite{Ono_2} avoids accumulating penalties until the end, then purposely crashes the population to maximize catch. Overall, our method yields a higher catch of $130.13$ units compared to $121.69$ units at $75\%$ safety.

\iffullversion
In comparison with \cite{Ono_2} our algorithm tends to be computationally slightly more efficient. The reason is that the policy associated with the minimum control cost in \eqref{eq_max_cost_evaluation} is applied until the end of the time horizon whenever $b_k$ is zero, which is independent of $\lambda$ and precomputed by the border case check. Thus, when $\lambda$ is updated via bisection, it suffices run the DP recursions over states which have $b_k=1$, which only exist within the safe set ${\mathcal{A}}$. In \cite{Ono_2}, on the other hand, the policy has to be reevaluated for all states in the state space $\mathcal{X}$ as $\lambda$ changes.
\fi



\end{example}


\section{Conclusions and future works}
\label{sec_conclusion}
\iffullversion
We consider the finite-time optimal control of stochastic systems subject to joint chance constraints and proposed a Dynamic Programming scheme to solve for the respective optimal policy.
Our analysis indicates that policies induced by Problem \eqref{eq_problemFormulation_CCOC} may be controversial for many applications. For instance, Problem \eqref{eq_problemFormulation_CCOC} just considers safety and cost in expectation, meaning over infinitely many trials starting from the initial state. This imposes the controller to deliberately fail on constraints to achieve a cost reduction in expectation, which is also in line with our introduction of a binary state and mixed policies. On the other hand, trying to avoid deliberative failure by adding a constant penalty for unsafe states, as implicitly done in \cite{Ono_2}, is not solving Problem \eqref{eq_problemFormulation_CCOC} to optimality and, furthermore, not necessarily generating a meaningful behaviour either. Problem \eqref{eq_problemFormulation_CCOC} is tailored to applications which reset the system state to the initial state after $N$ time-steps, played infinitely often, and where constraints become redundant once they are violated. Other applications might require a different problem formulation, whose exploration is subject of future work. 
\fi
\rev{Expanding on our results, LP formulations are proposed in \cite{schmid2024joint}. Beyond gridding, this allows to approximate value functions via linear combinations of basis functions. Exploring scalable frameworks like MPC, to which the approach in \cite{Ono_2} easily extends (see, e.g.,  \cite{ono_5}), would further be interesting.}

\iffullversion
\section{Appendix}
\subsection{Duality theory}\label{sec_duality}
Consider the optimization problem
\begin{equation}\label{eq:primal}
\begin{aligned}
f^\star=\inf_{x \in \mathbb{R}^n} & \quad f(x)\\
\text{subject to} &\quad g_j(x) \leq 0, \quad j = 1,\ldots,r,
\end{aligned}
\end{equation}
where $f:\mathbb{R}^n \rightarrow \mathbb{R}$, $g_j : \mathbb{R}^n \rightarrow \mathbb{R}$. We refer to \eqref{eq:primal} as the primal problem and we denote its value by $f^\star$. The dual problem is given by
\begin{equation}\label{eq:dual}
\begin{aligned}
q^\star=\max_{\lambda_1,\dots,\lambda_r} & \quad q(\lambda)\\
\text{subject to} & \quad \lambda_j \geq 0, \quad j = 1,\ldots,r,
\end{aligned}
\end{equation}
where
$
q(\lambda) = \inf_{x \in \mathbb{R}^n} \{ f(x) + \sum_{j=1}^r \lambda_j g_j(x)\}$ is called the dual function. The function $L(x,\lambda) = f(x) + \sum_{j=1}^r \lambda_j g_j(x) $ is referred to as the Lagrangian. The dual problem is always a convex optimization problem even if the primal
is not convex \cite{boyd_1}. In general, $q^\star \geq f^\star$; if $q^\star = f^\star$ we say that strong duality holds and there is no duality gap. For a given $\overline{\lambda} \in \mathbb{R}^r$, let 
$
x_{\overline{\lambda}} \in \argmin_{x \in \mathbb{R}^n} L(x,\overline{\lambda})
$ be a value minimizing the Lagrangian.
Then, $g(x_{\overline{\lambda}})$ is a subgradient of the dual function $q$ evaluated at $\overline{\lambda}$. The duality theory discussed above can be extended to infinite dimensional spaces, see \cite{Anderson_1}. 
\fi

\bibliographystyle{IEEEtran}
\bibliography{references}

\end{document}